\documentclass[letterpaper,10pt,reqno,onefignum,onetabnum]{amsart}
\usepackage[english]{babel}
\usepackage{amsmath}
\usepackage{amsthm}
\usepackage{verbatim}
\usepackage[foot]{amsaddr}
\usepackage{delimset}
\usepackage[left]{lineno}
\usepackage[dvipsnames]{xcolor}
\usepackage{hyperref}
\hypersetup{
	colorlinks = true,
	linkcolor = OliveGreen,
	anchorcolor = OliveGreen,
	citecolor = OliveGreen,
	filecolor = OliveGreen,
	urlcolor = OliveGreen
}
\usepackage{algorithm}
\usepackage{algorithmic}
\usepackage{float}
\usepackage{lipsum}
\usepackage{amsfonts}
\usepackage{amssymb}
\usepackage{graphicx}
\usepackage{epstopdf}
\usepackage{cancel}
\usepackage{cases}
\usepackage{multirow}
\ifpdf
\DeclareGraphicsExtensions{.eps,.pdf,.png,.jpg}
\else
\DeclareGraphicsExtensions{.eps}
\fi

\newtheorem{teo}{Theorem}[section]
\newtheorem{prop}[teo]{Proposition}
\newtheorem{lem}[teo]{Lemma}

\newtheorem{pro}[teo]{Problem}
\newtheorem{algo}[teo]{Algorithm}

\newtheorem{rem}[teo]{Remark}
\newtheorem{example}[teo]{Example}

\usepackage{algorithmic}
\usepackage{tikz}
\usepackage{booktabs}%
\usetikzlibrary{matrix}
\usetikzlibrary{arrows}
\usepackage{mathrsfs}
\usepackage{bm}
\usepackage{color}
\usepackage{xcolor}
\usepackage{subcaption}

\newcommand{\N}{\mathbb N}

\newcommand{\R}{\mathbb R}

\renewcommand{\H}{\mathcal{H}}
\newcommand{\G}{\mathcal G}

\newcommand{\HH}{{\bm{\mathcal{H}}}}

\newcommand{\id}{\textnormal{Id}}

\newcommand{\lev}{\textnormal{lev}}
\newcommand{\weak}{\rightharpoonup}
\newcommand{\ran}{\textnormal{ran}\,}
\newcommand{\dom}{\textnormal{dom}\,}

\newcommand{\interior}{\textnormal{int}}
\newcommand{\zer}{\textnormal{zer}}

\newcommand{\gra}{\textnormal{gra}\,}

\newcommand{\scal}[2]{{\left\langle{{#1}\mid{#2}}\right\rangle}}

\newcommand{\menge}[2]{\big\{{#1}~\big |~{#2}\big\}}

\newcommand{\RP}{\ensuremath{\left[0,+\infty\right[}}
\newcommand{\RM}{\ensuremath{\left]-\infty,0\right]}}
\newcommand{\RMM}{\ensuremath{\left]-\infty,0\right[}}

\newcommand{\RPP}{\ensuremath{\left]0,+\infty\right[}}

\newcommand{\RX}{\ensuremath{\left]-\infty,+\infty\right]}}

\newcommand{\sri}{\ensuremath{\text{\rm sri}\,}}
\newcommand{\inte}{\ensuremath{\operatorname{int}}}

\newcommand{\prox}{\ensuremath{\text{\rm prox}\,}}
\newcommand{\conv}{\ensuremath{\text{\rm conv}\,}}

\usepackage{geometry}
\numberwithin{equation}{section}

\numberwithin{equation}{section}

\DeclareFontEncoding{FMS}{}{}
\DeclareFontSubstitution{FMS}{futm}{m}{n}
\DeclareFontEncoding{FMX}{}{}
\DeclareFontSubstitution{FMX}{futm}{m}{n}
\DeclareSymbolFont{fouriersymbols}{FMS}{futm}{m}{n}
\DeclareSymbolFont{fourierlargesymbols}{FMX}{futm}{m}{n}
\DeclareMathDelimiter{\nr}{\mathord}{fouriersymbols}{152}{fourierlargesymbols}{147}

\DeclareMathOperator*{\argmin}{arg\,min}
\DeclareMathDelimiter{\nr}{\mathord}{fouriersymbols}{152}{fourierlargesymbols}{147}
\DeclareMathAlphabet{\mathpzc}{OT1}{pzc}{m}{it}

\title[Forward-Reflected-Backward with Linesearch]{Forward-Reflected-Backward algorithm with Linesearch}

\author{Fernando Muñoz García$^{1}$}
\author{Fernando Rold\'an$^{2}$}
\address{$^{1}$ Departamento de Ingeniería Matemática, Universidad de Concepción, Concepción, Chile.}
\address{$^{2}$ Departamento de Ingeniería Matemática and CI$^2$MA, Universidad de Concepción, Concepción, Chile.}
\email{femunoz2022@udec.cl}
\email{fernandoroldan@udec.cl}
\begin{document}
	\begin{abstract}
      In this article, we aim to solve a monotone inclusion problem involving the sum of a maximally monotone operator and a continuous operator. While several algorithms exist to solve this problem when the continuous operator is cocoercive or Lipschitz continuous, they typically require the estimation of the global Lipschitz constant, which can be computationally expensive and often imposes overly restrictive step-sizes. To avoid these limitations and to handle merely continuous operators, linesearch subroutines are employed. A popular method in this context is the forward-backward-forward (FBF) algorithm (also known as Tseng's splitting), which utilizes a linesearch to guarantee convergence. However, a drawback of FBF is that the continuous operator must be evaluated twice per iteration. On the other hand, the forward-reflected-backward (FRB) algorithm proposed by Malitsky and Tam (2020) requires only a single evaluation of the operator per iteration. Although a linesearch version of FRB exists, its convergence is guaranteed only for locally Lipschitz operators; in fact, we present an example demonstrating that this existing linesearch can fail to terminate when the operator is merely continuous. In this work, we propose a novel linesearch strategy for FRB that is well defined and guarantees convergence even when the operator is merely continuous. We also extend the proposed algorithm to handle additional cocoercive and Lipschitz continuous operators. Finally, we provide numerical experiments on saddle-point problems and image restoration. The numerical results show that FRB with the proposed linesearch can accelerate the numerical convergence even when the operator is Lipschitz continuous. In addition, these results show that the proposed method is competitive with linesearch FBF, offering considerable computational advantages in various scenarios.
		\par
		\bigskip
		
	\end{abstract}
		\maketitle
	\section{Introduction} 
Monotone inclusion problems model several applications, such as variational inequalities and optimization \cite{bauschkebook2017,Combettes2018MP}, equilibrium problems \cite{Combettes2005equilibrium}, partial differential equations \cite{AubinHelene2009,Glowinsky1975,Showalter1997}, signal processing and imaging \cite{BotHendrich2014TV,Briceno2011ImRe,BurgerSawatzkySteidl2014,chambolle2016AN}, traffic theory \cite{Nets1,GafniBert84}, machine learning \cite{BotrelaxFBF2023,Nocedal2018,CombettesPesquet2021strategies}, among others. In this article, we aim to numerically solve a monotone inclusion problem in a real Hilbert space $\H$. In particular, the problem involves the sum of a maximally monotone operator $A \colon \H \to 2^\H$ and a monotone and continuous operator $B\colon \H \to \H$.  When $B=0$, this problem can be solved by the proximal point algorithm \cite{martinet1970,RockafellarSIAM1976}. When the operator $B$ is cocoercive, the problem can be solved by the forward-backward (FB) algorithm \cite{Chenhg1997,passty1979JMAA}. However, the FB algorithm can fail to converge when $B$ is merely Lipschitz continuous (for instance, considering $A=0$ and $B$ as a rotation operator in $\R^2$). To address the case where $B$ is Lipschitz continuous, Tseng proposed the forward-backward-forward (FBF) algorithm by including an additional forward step on $B$ \cite{Tseng2000SIAM}. Nevertheless, these methods require the knowledge of the cocoercivity or Lipschitz constant to initialize the routine. Estimating these constants can be problematic in itself, often leading to a computationally costly task. Moreover, the global constraints imposed on the step-size by these parameters may be significantly more restrictive than what the algorithm could allow locally during the iterations. In addition, not every continuous operator is Lipschitz continuous. In view of these facts, Tseng in \cite{Tseng2000SIAM} proposed a linesearch strategy to find an adequate step-size at each iteration, guaranteeing the convergence of FBF when $B$ is continuous and not necessarily Lipschitz continuous. The FBF method was combined with FB in \cite{BricenoDavis2018} to solve inclusion problems that additionally involved cocoercive operators; this method was called forward-backward-half-forward (FBHF) because the cocoercive operator requires only one evaluation at each step. Later, in \cite{BricenoRoldan4op}, the authors generalized the aforementioned works by proposing a method to solve monotone inclusions involving the sum of four operators, incorporating Lipschitz continuous operators as well. 

A major drawback of the FBF algorithm is the second activation of the operator $B$, which increases the computational cost of each iteration and becomes unfavorable in large-scale problems. Recently, new algorithms for solving the problem in the Lipschitz continuous setting have been proposed \cite{Cevher2020SVVA,Csetnek2019AMO,Malitsky2020SIAMJO}. These methods require only one evaluation of $B$ per iteration. In particular, the authors in \cite{Malitsky2020SIAMJO} proposed the forward-reflected-backward (FRB) method, which avoids a second activation of $B$ in the current step by storing its evaluation from the previous iterate. A version of FRB with linesearch was also proposed in \cite{Malitsky2020SIAMJO}, but its convergence was guaranteed only for locally Lipschitz operators. Consequently, the convergence of FRB for merely continuous operators remains an open problem. Additionally, a variant of FRB considering cocoercive operators, called forward-half-reflected-backward (FHRB), was also proposed in \cite{Malitsky2020SIAMJO}. Several recent papers extend the FRB framework: incorporating inertia and momentum \cite{Adeolu2023,DungNguyenFRB2026,RoldanVega2025,Yao2024,Zhang2025}, applying variance reduction techniques \cite{Tam2021,Tran-Dinh2025}, combining it with the Douglas--Rachford algorithm \cite{DuNguyen2025,lionsandmercier1979,RyuVu2020}, coupling it with the Partial Inverse method \cite{Roldan2024,Spingarn1985MP}, and extending it to nonlinear forward-backward methods with momentum correction \cite{MorinBanertGiselsson2022}, to name a few.

In this article, we propose a linesearch version of FRB for solving our main problem. In particular, we present a counterexample where the operator $B$ is not locally Lipschitz and the linesearch for FRB proposed in \cite{Malitsky2020SIAMJO} fails to stop. To overcome this limitation, we propose a novel linesearch strategy that also incorporates the image of $B$ from previous iterates, making the linesearch condition feasible. We further extend the proposed method to a four-operator splitting framework, allowing us to solve problems that additionally involve both cocoercive and Lipschitz continuous operators. As a consequence, we derive a splitting algorithm for solving linear composite optimization problems with nonlinear constraints. Finally, we present numerical experiments on saddle-point problems and image restoration to demonstrate the numerical advantages of the proposed method.

The remainder of the paper is organized as follows. Section~\ref{sec:pre} presents our notation and preliminary results. In Section~\ref{sec:probformluation}, we detail the problem formulation, provide a rigorous analysis of the existing methods in the literature, and present the counterexample for which FRB with linesearch for locally Lipschitz operators fails to terminate. We derive our main convergence results in Section~\ref{sec:MN} and extend them to include cocoercive and Lipschitz continuous operators in Section~\ref{sec:cocoandlips}. The numerical experiments are presented in Section~\ref{sec:num}. Finally, Section~\ref{sec:conclu} is dedicated to our concluding remarks.

\section{Preliminaries} \label{sec:pre}
In this article, $\H$ and $\G$ are real Hilbert spaces endowed with inner
product $\scal{\cdot}{\cdot}$ and induced norm $\|\cdot \|=\sqrt{\scal{\cdot}{\cdot}}$. Given $x \in \H$ and  $\delta \in \RPP$, $\mathcal {B}(x,\delta)$ denotes the ball centered at $x$ of radius $\delta$.
We denote
the strong and weak 
convergence by $\to$ and $\weak$, respectively. The identity operator is 
denoted by $\id$.  The
following classic inequalities will be used throughout the article:
\begin{eqnarray}\label{eq:CS}
	&\hspace*{-1cm}(\textnormal{Cauchy–Schwarz inequality})  &(\forall 
	(x,u) \in 
	\H^2) \quad 
	|\scal{x}{u}| \leq \|x\|\|u\|,\\
	\label{eq:YI}&(\textnormal{Young's inequality})   &(\forall (a,b)\in 
	\R^2)(
	\forall \epsilon 
	>0) \quad 2ab \leq 
	\epsilon a^2+\frac{b^2}{\epsilon}.
\end{eqnarray} 	
Let $D\subset \H$, $T\colon D \rightarrow 
\H$, and $\beta \in \left]0,+\infty\right[$. The operator $T$ is 
$\beta$-cocoercive if 
\begin{equation} \label{def:coco}
	(\forall (x,y) \in D^2) \quad \langle x-y \mid 
	Tx-Ty 
	\rangle 
	\geq \beta \|Tx - Ty 
	\|^2.
\end{equation}
We say that $T$ is $\beta$-Lipschitz continuous if 
\begin{equation} \label{def:lips}
	(\forall x \in D) (\forall y \in D)\quad \|Tx-Ty\| \leq  
	\beta\|x - y \|.
\end{equation}	
Note that if $T$ is $\beta$-cocoercive it is $(1/\beta)$-Lipschitz continuous. The operator $T$ is continuous on $D$ if for any $u \in D$ and any sequence $(u_n)_{n \in \N}$ in $D$ such that $\|u_n - u\| \to 0$, one has $\|Tu_n - Tu\| \to 0$. Moreover, the operator $T$ is said to be uniformly continuous on $D$ if for any two sequences $(u_n)_{n \in \N}$ and $(v_n)_{n \in \N}$ in $D$ such that $\|u_n - v_n\| \to 0$, one has $\|Tu_n - Tv_n\| \to 0$.

We denote the power set of $\H$ by $2^\H$. Let $A\colon\H \rightarrow 2^{\H}$ be a set-valued operator.
The domain, range, zeros, and graph of $A$ 
are defined, respectively, by:
\begin{align*}
	&\dom\, A = \menge{x \in \H}{Ax \neq  \varnothing}, \quad 
	\ran\, A = \menge{u \in \H}{(\exists x \in \H)\,\, u \in Ax}\\
	&\zer A = 	\menge{x \in \H}{0 \in Ax}, \quad
	\gra A = \menge{(x,u) \in \H \times \H}{u \in Ax}.
\end{align*}
The inverse of $A$ is defined by
$A^{-1}\colon\H \rightarrow 2^{\H} \colon u \mapsto \menge{x \in \H}{u \in Ax}$.
The operator $A$ is called monotone if 
\begin{equation}\label{eq:defrhomon}
	\left(\forall \big((x,u),(y,v)\big) \in (\gra A)^2\right) 
	\quad 
	\scal{x-y}{u-v} 
	\geq 0.
\end{equation}
Moreover, $A$ is maximally monotone if it is 
monotone and its graph is 
maximal in the sense of 
inclusions among the graphs of monotone operators. We say that $A$ is locally bounded at $x \in \H$, if there exists $\delta \in \RPP$ such that $A(\mathcal {B}(x,\delta))$ is bounded. Given $D\subset \H$ nonempty, $A$ is locally bounded at $D$ if, for every $x \in D$, $A$ is locally bounded at $x$. Let $A \colon \H \to 2
^\H$ be a maximally monotone operator, then, $\overline{\dom} A$ is a convex subset of $\H$. Moreover, $A^{-1}$ is also a maximally monotone operator. The resolvent of $A$ is the single valued operator
defined by $J_A\colon\H \rightarrow 2^{\H} \colon x \mapsto (\id+A)^{-1}x$. 

We denote by $\Gamma_0(\H)$ the class of proper lower 
semicontinuous convex functions $f\colon\H\to\RX$. Let 
$f\in\Gamma_0(\H)$. Given $t \in \R$, the level set of $f$ at $t$ is denoted by $\textnormal{lev}_{\leq t} f$. The Fenchel conjugate of $f$ is 
defined by 
\begin{equation*}
    f^*\colon u\mapsto \sup_{x\in\H}(\scal{x}{u}-f(x)).
\end{equation*}
We have
$f^*\in \Gamma_0(\H)$. The subdifferential of $f$ is the set valued operator defined by
\begin{equation*}
    \partial f\colon x\mapsto \menge{u\in\H}{(\forall y\in\H)\:\: 
	f(x)+\scal{y-x}{u}\le f(y)}.
\end{equation*}
We have that $\partial f$ is a maximally monotone operator,
$(\partial f)^{-1}=\partial f^*$,
and that $\zer\,\partial f$ is the set of 
minimizers of $f$, which is denoted by $\argmin_{x\in \H}f$. 	
The proximity operator of $f$ is defined by
\begin{equation*}
	\prox_{f}\colon 
	x\mapsto\argmin_{y\in\H}\left(f(y)+\frac{1}{2}\|x-y\|^2\right)
    \end{equation*}
and we have $\prox_f=J_{\partial f}$. For a further background on monotone operators and convex analysis, the reader is referred to \cite{bauschkebook2017}. We conclude this section with the following lemma.
    
\begin{lem}[Lemma~3.2~\cite{BricenoRoldan4op}]\label{lemma}
Let $A\colon \H \to 2^\H$ be a maximally monotone operator, let $B\colon \H \to \H$ be an operator such that $\dom A \subset \dom B$ and $B$ is continuous in $\dom A$. Moreover, let $z \in \dom A$,  let $y$ in 
$\H$, and define 
\begin{equation}\label{eq:defxzy}
(\forall \gamma > 0) \quad \quad	x_{z,y}(\gamma) = 
J_{\gamma A} 
( z - 
\gamma y). 
\end{equation}
Then, the following assertions hold
\begin{enumerate}
    \item\label{lemmai} The function 
    $$\gamma \mapsto \frac{1}{\gamma} \|z- x_{z,y}(\gamma)\|$$
    is nonincreasing.
    \item\label{lemmaii} For every 
$\theta \in \left]0,1\right[$, there exists $\gamma(z) >0$ 
such 
that, for 
every  $\gamma \in \left]0, \gamma(z)\right]$,
\begin{equation}\label{eq:destheta}
\gamma \| Bz - B x_{z,y} (\gamma) \| \leq \theta \|z 
- x_{z,y} 
(\gamma)\|.
\end{equation}
\end{enumerate}
\end{lem}
\section{Problem Formulation}\label{sec:probformluation}
In this section we present the main problem and a review of some existing methods in the literature to solve it. We aim at solving the following monotone inclusion problem.
\begin{pro}\label{prob:problemMAIN}
Let $\H$ be a real Hilbert space, let $A: \H \to 2^\H$ be a maximally monotone operator, and let $B : \H \to 2^\H$ be a 
maximally 
monotone operator. Suppose that $\overline{\dom} A \subset \dom B$, that $B$ is single valued and continuous in $\overline{\dom} A$, and that $A+B$ is maximally monotone. The problem is to 
\begin{equation}\label{eq:problem1}
\text{find} \quad x \in \H \quad \text{such that} \quad 0 \in 
Ax+Bx,
\end{equation}
under the assumption that the set of solutions to 
\eqref{prob:problemMAIN} is nonempty.
\end{pro}
This problem can be solved by the Tseng's forward-backward-forward (FBF) algorithm \cite{Tseng2000SIAM}. Given $x_0 \in \H$ and a nonnegative sequence $(\gamma_n)_{n \in \N}$, FBF iterates as follows 
\begin{equation}
			\label{eq:algoFBF}
			(\forall n \in \N)\quad \left\lfloor
			\begin{aligned}
                &z_n = J_{\gamma_n A}(x_n-\gamma_nBx_n)\\
                &x_{n+1} = z_n-\gamma_n (Bz_n -Bx_n).			\end{aligned}
			\right.
\end{equation} 
The sequence $(\gamma_n)_{n \in \N}$ is chosen according to a subroutine called {\it linesearch}. For each $n \in \N$, this subroutine is defined as follows: given $\tau\in \RPP$, $\sigma \in~]0,1[$, and $\theta \in~]0,1[$, $\gamma_n \in \RPP$ is the largest value in $\{\tau, \tau\sigma, \tau \sigma^2, \tau \sigma^3,\ldots  \}$ such that
\begin{equation}\label{eq:linesearchFBF}
    \gamma \|Bz(\gamma) - Bx_n \| \leq \theta \|z(\gamma)-x_n\|,
\end{equation}
where
\begin{equation}
    z(\gamma) = J_{\gamma A}(x_n-\gamma Bx_n).
\end{equation}
In the case where $B$ is $\beta$-Lipschitz continuous, one can simply choose $\gamma_n \equiv \gamma \in ]0,1/\beta[$. The FBF algorithm was extended in \cite{BricenoDavis2018} and \cite{BricenoRoldan4op} for solving inclusions involving not only maximally monotone operators, but also cocoercive and Lipschitz. 
Note that FBF requires two activations of $B$ at each iteration.
If the operator $B$ is cocoercive, Problem~\ref{prob:problemMAIN} can be solved by the forward-backward (FB) algorithm \cite{Chenhg1997,passty1979JMAA}. If $B$ is $\beta$-Lipschitz continuous, methods requiring only one activation of $B$ per iteration were presented in \cite{Cevher2020SVVA,Csetnek2019AMO,Malitsky2020SIAMJO}. Particularly, the authors in \cite{Malitsky2020SIAMJO} proposed the method called forward-reflected-backward (FRB) which, for $(x_0,x_{-1}) \in \H^2$ and a sequence of step-sizes $(\gamma_n)_{n \in \N}$, iterates as follows:
\begin{equation}
			\label{eq:algoMTintro}
			(\forall n \in \N)\quad \left\lfloor
			\begin{aligned}
                &v_n = Bx_n -Bx_{n-1}\\
                &x_{n+1} = J_{\gamma_n A}(x_n-\gamma_n Bx_n -\gamma_{n-1}v_n).			\end{aligned}
			\right.
\end{equation} 
The weak convergence of FRB to a solution to Problem~\ref{prob:problemMAIN} is ensured when $\gamma \in ]0,1/(2\beta)[$.
In addition, a linesearch version of FRB is also proposed in \cite{Malitsky2020SIAMJO}, whose convergence is guaranteed in the case that $B$ is locally Lipschitz. This linesearch is similar to those proposed by Tseng: $\gamma_n \in \RPP$ is the largest value in $\{ \gamma_{n-1}, \gamma_{n-1}\sigma, \gamma_{n-1} \sigma^2,\ldots  \}$ such that
\begin{equation}\label{eq:linesearchFRB}
    \gamma \|Bx(\gamma) - Bx_n \| \leq \frac{\theta}{2} \|x(\gamma)-x_n\|,
\end{equation}
where
\begin{equation}
    x(\gamma) = J_{\gamma A}(x_n-\gamma Bx_n-\gamma_{n-1}v_n).
\end{equation}
Note that, the convergence of FRB has not been proved in the case where $B$ is merely continuous. The following example illustrates that the proposed linesearch for FRB may never stop. This example was generated with the assistance of Google's Gemini AI (Gemini 1.5 Pro).
\begin{example}\label{example1} In the context of Problem~\ref{prob:problemMAIN} set  $\mathcal{H} = \mathbb{R}^2$, let $C \subset \R^2$ given by
	\begin{equation}
    C = \menge{(x,y) \in \mathbb{R}^2}{y \geq |x|^{3/2}},
	\end{equation}
set $A = \partial \delta_C$, and let $B \colon \mathbb{R}^2 \to \mathbb{R}^2$ be the operator defined by:
	\begin{equation}
		B \colon \R^2 \to \R^2 \colon (x,y)  \mapsto \left( \operatorname{sgn}(x)|x|^{1/3} + 1, y \right).
	\end{equation}
	Since $t \mapsto \operatorname{sgn}(t)|t|^{1/3}$ is strictly increasing and continuous, $B$ is maximally monotone. However, $B$ is not locally Lipschitz at $(0,0)$. In this case, \eqref{eq:problem1} corresponds to
    \begin{equation}\label{eq:incluexample}
        \text{find} \quad (x,y) \in \R^2 \quad \text{such that} \quad (0,0) \in 
N_{\textnormal{lev}_{\leq 0}  f}(x,y)+\left( \operatorname{sgn}(x)|x|^{1/3} + 1, y \right),
    \end{equation}
    where $f \colon (x,y) \mapsto  |x|^{3/2}-y$. Moreover, we have $(u,v) \in N_{{\textnormal{lev}_{\leq 0}  f}}(x,y)$ if and only if, there exists $\lambda \in \RPP$, 
\begin{equation}\label{eq:NCexample}
     (u,v) =\lambda\begin{cases}
     (\frac{3}{2}\operatorname{sgn}(x)|x|^{1/2},-1), &\textnormal{ if } y = |x|^{3/2} \\
     (0,0), &\textnormal{ if } y > |x|^{3/2}.
    \end{cases}
\end{equation}
    Therefore, \eqref{eq:incluexample} reduces to finding $(x,y) \in \R^2$ such that 
    $y = |x|^{3/2}$ and
        \begin{equation}\label{eq:incluexample2}
        0  =3\operatorname{sgn}(x)|x|^2+2\operatorname{sgn}(x)|x|^{1/3} +2.
    \end{equation}
    Note that \eqref{eq:incluexample2} has no solution if $x \in \RP$. Otherwise, if $x \in \RMM$, \eqref{eq:incluexample2} is equivalent to $0=3|x|^2+2|x|^
    {1/3}-2$ which has a unique solution for $x \in \RMM$. We conclude that, the problem has a unique solution.

    Consider now the following initialization for FRB:
	\begin{equation*}
		x_{-1} = (1, -1), \quad x_0 = (0, -1), \quad \text{and} \quad \gamma_{-1} = \gamma_0 = 1.
	\end{equation*}
	For $n=0$ we have
	\begin{align*}
		v_0 &= Bx_0 - Bx_{-1} = (1, -1) - (2, -1) = (-1, 0), \\
		z_0 &= x_0 - \gamma_0 Bx_0 - \gamma_{-1}v_0 = (0, -1) - (1, -1) - (-1, 0) = (0, 0).
	\end{align*}
    In addition, the next iterate is $x_1 = P_C(z_0)$. Since $(0,0) \in C$, we obtain 
	\begin{align*}
    x_1 &= (0,0)\\
		v_1 &= Bx_1 - Bx_0 = (1, 0) - (1, -1) = (0, 1).
	\end{align*}
	Given $\gamma> 0$ we have
	\begin{equation}
		z(\gamma):=x_1 - \gamma Bx_1 - \gamma_0 v_1 = (0,0) - \gamma(1,0) - (0,1) = (-\gamma, -1).
	\end{equation}
	Let $x(\gamma) = P_C(z(\gamma))$. For $\gamma > 0$, the projection lies on the boundary of $C$ and we can parameterize this exact projection for $t > 0$ as $x_t = (-t^2, t^3)$. Then,  $x(\gamma) = P_C(z(\gamma))$ is equivalent to $z(\gamma) - x_t \in N_C(x_t)$, which according \eqref{eq:NCexample}, reduces to
	\begin{equation}
		(-\gamma + t^2, -1 - t^3) = -\lambda \left( \frac{3}{2}t, \; 1 \right)
	\end{equation}
    for some $\lambda \in \RPP$. Note that $\lambda = (1 + t^3)$, thus, this system can be reduced to
    	\begin{equation}\label{eq:gamma_v}
		\gamma = \frac{3}{2}t^4 + t^2 + \frac{3}{2}t.
	\end{equation}
	Since $\phi \colon t \mapsto \frac{3}{2}t^4 + t^2 + \frac{3}{2}t$ is strictly increasing for $t > 0$ and $\phi(0) = 0$, we will study the linesearch in \eqref{eq:linesearchFRB} for    
    $\gamma(t):= \phi (t)$. Let us now evaluate the terms involved in the linesearch: 
	\begin{align*}
		\|x_{\gamma(t)} - x_1\| &= \sqrt{(-t^2)^2 + (t^3)^2} = t^2 \sqrt{1 + t^2},\\
		\|Bx_{\gamma(t)} - Bx_1\| &= \left\| \left(-t^{2/3} + 1, t^3\right) - (1, 0) \right\| = t^{2/3} \sqrt{1 + t^{14/3}}.
	\end{align*}
	Therefore, the linesearch condition is
	\begin{equation}
		\left( \frac{3}{2}t^4 + t^2 + \frac{3}{2}t \right) t^{2/3} \sqrt{1 + t^{14/3}} \leq \frac{\theta}{2} t^2 \sqrt{1 + t^2}.
	\end{equation}
	Dividing both sides by $t^{5/3}$, we obtain:
	\begin{equation}
		\left( \frac{3}{2}t^3 + t + \frac{3}{2} \right) \sqrt{1 + t^{14/3}} \leq \frac{\theta}{2} t^{1/3} \sqrt{1 + t^2}.
	\end{equation}
	Taking $t \downarrow 0$ (corresponding to $\gamma \downarrow 0$), the left-hand side converges to $3/2$ while the right-hand side converges to $0$, showing that the linesearch procedure will never stop in this setting. To illustrate this behavior, Figure~\ref{fig:linesearch} plots the left-hand side (LHS) and right-hand side (RHS) of the linesearch condition \eqref{eq:linesearchFRB} for this example, with parameters $\theta = 0.9$, $\tau=1$, and $\sigma = 0.9$. This plot exhibits the asymptotic behavior derived in our calculations.
    \begin{figure}
        \centering
        \includegraphics[width=0.5\linewidth]{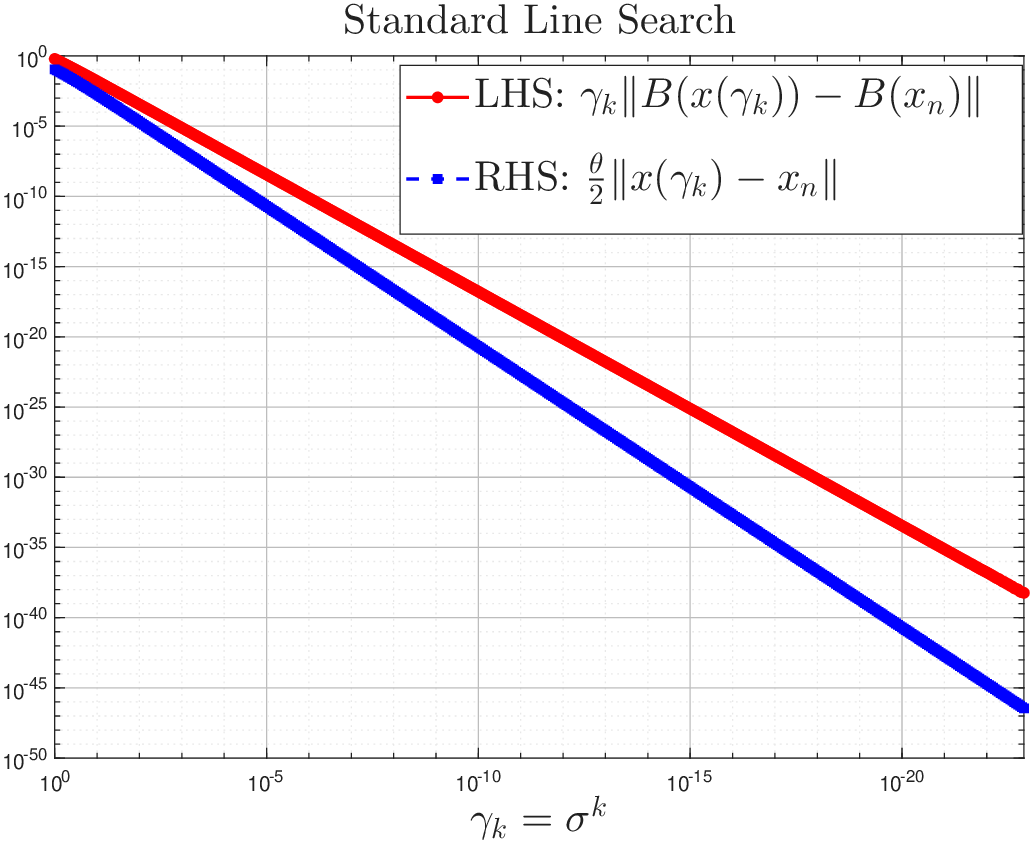}
        \caption{}
        \label{fig:linesearch}
    \end{figure}
\end{example}
\begin{rem} The author in \cite{Tseng2000SIAM} proves that the linesearch in \eqref{eq:linesearchFBF} for FBF is satisfied with $\gamma = \tau$ when $x_n \in \zer(A+B)$ \cite[Theorem~3.4]{Tseng2000SIAM}. In addition, the author shows that if 
\begin{equation}
  (\forall \gamma \in \RPP) \quad  \gamma \|Bz(\gamma) - Bx_n \| > \theta \|z(\gamma)-x_n\|,
\end{equation}
   then necessarily $x_n \in \zer(A+B)$, guaranteeing that the linesearch is well defined and terminates in finitely many steps. The proof relies on fixed-point characterization of the forward-backward operator: $x_n = J_{\gamma A}(x_n-\gamma Bx_n)$ if and only if $x_n \in \zer (A+B)$. On the other hand, the momentum term $v_n$ in FRB breaks this property. If $x_n = J_{\gamma A}(x_n-\gamma Bx_n-\gamma_{n-1}v_n)$, this equality no longer characterizes the elements of $\zer(A+B)$. This structural difference suggests that the term $v_n$ must be explicitly considered in the linesearch condition for FRB, which motivates the approach developed in the next section. Note that this fact was also mentioned in \cite[Concluding remarks]{Malitsky2020SIAMJO} where the auxiliary variable $u_{n+1}=Bx_n$ is needed to interpret FRB as a fixed-point iteration.
\end{rem}
    \section{Main Results}\label{sec:MN}
    In this section we present our main algorithm and the result of convergence.
\begin{algo}\label{alg:algoMT}
		In the context of Problem~\ref{prob:problemMAIN}, let $\gamma_{-1} \in \RPP$, let $(\theta,\sigma) \in ~ ]0,1[^2$, let $\tau \in \RPP$, let $(\alpha_k)_{k \in \N}$ be defined by $\alpha_k =\tau \sigma^k$, and let $(x_0,x_{-1}) \in (\dom A)^2$. Define $v_0 = Bx_0-Bx_{-1}$ and consider the sequence defined recursively by
		\begin{equation}
			\label{eq:algoMT}
			(\forall n \in \N)\quad \left\lfloor
			\begin{aligned}
                & v_{n}  = Bx_{n} - B x_{n-1}\\
                &x_{n+1} = J_{\gamma_n A} \left( x_n -\gamma_n B x_n -\gamma_{n-1}v_n\right),
			\end{aligned}
			\right.
		\end{equation} 
        where, for each $n \in \N$, $\gamma_n = \alpha_k$ where $k$ is the smallest natural number such that 
        \begin{equation}\label{eq:linesearch}
            \gamma_{n}\|v_{n+1}\| \leq \theta (\|x_{n+1}-x_n\|+\gamma_{n-1}\|v_n\|).
        \end{equation}
	\end{algo}	
The following proposition guarantees that Algorithm~\ref{alg:algoMT} is well defined, that is, the linesearch procedure finishes in a finite number of steps.
\begin{prop}\label{prop:Glinesearch}
In the context of Problem~\ref{prob:problemMAIN}, let $(v_n,x_n)_{n \in \N}$ be generated by Algorithm~\ref{alg:algoMT}. Then, for each $n \in \N$, there exists $k_n \in \N$ such that \eqref{eq:linesearch} holds for $\gamma_n = \alpha_{k_n}$.
\end{prop}
\begin{proof} Fix $n \in \N$, set $w_n = x_n-\gamma_{n-1}v_n$, and, for each $k \in \N$, define
\begin{equation}\label{eq:linesearchprop1}
   			\begin{aligned}
                & z_k =  J_{\alpha_k A} \left( w_n -\alpha_k B x_n\right), \\
                & p_k = Bz_k - B x_n.        
			\end{aligned}
\end{equation}
We will prove that there exists $k \in \N$ such that
\begin{equation}\label{eq:linesearchprop2}
  (\forall k \geq k_n) \quad \alpha_k\| p_k \| \leq \theta (\|z_k-x_n\|+\gamma_{n-1}\|v_n\|).
 \end{equation}
First, note that, for each $k \in \N$, $z_k \in \dom A \subset \dom B$. In the case where $v_n = 0$, the result follows from Lemma~\ref{lemma}~\ref{lemmaii} applied to $z = x_n$ and $y = Bx_n$. Suppose now that $v_n \neq 0$. Since $\alpha_k \to 0$, we have that $z_k \to \widetilde{z}:=P_{\overline{\dom} A }w_n$
\cite[Theorem~23.48]{bauschkebook2017}. Since $B$ is continuous in $\overline{\dom} A$, it is locally bounded on $\overline{\dom} A$, thus, the sequence $(Bz_k)_{k \in \N}$ is bounded. Therefore, $\alpha_k \to 0$ and $p_k = Bz_k - Bx_n$ imply that $\alpha_k \|p_k\| \to 0$.  The result follows from the fact that $\theta \gamma_{n-1} \|v_n\| >0$.
\end{proof}
\begin{prop}\label{prop:MTFejer} In the context of Problem~\ref{prob:problemMAIN}, let $(v_n,x_n)_{n \in \N}$ be generated by Algorithm~\ref{alg:algoMT}. In addition, let $\widehat x \in \zer(A+B)$ and define, for each $n \in \N$,
\begin{equation}\label{eq:defGamma}
    \Gamma_n(\widehat{x}) = \|x_{n}-\widehat{x}\|^2 - 2\gamma_{n-1}\scal{x_n-\widehat{x}}{v_n}+\frac{\gamma_{n-1}^2}
    {\rho^2} \|v_n\|^2.
\end{equation}
Then, the following assertions hold:
\begin{enumerate}
    \item\label{prop:MTFejer1} For every $n \in \N$
        \begin{equation}\label{eq:desGamma}
         \Gamma_{n+1}(\widehat{x}) \leq \Gamma_{n}(\widehat{x}) -\left(1-\rho -\left(1+\frac{1}{\rho^2}\right)\theta^2\right)
     \left(\|x_{n+1}-x_n\|^2+\frac{\gamma_{n-1}^2}{\rho^2}\|v_{n}\|^2\right).
        \end{equation}
    \item \label{prop:MTFejer2} For every $n \in \N$
    \begin{equation}\label{eq:desGammabel} \Gamma_{n}(\widehat x) \geq \left(1-\rho^2\right) \|x_n-\widehat{x}\|^2.
    \end{equation}
    \end{enumerate}
\end{prop}
\begin{proof}  \begin{enumerate}
\item First, note that $(\gamma_n)_{n \in \N}$ is well defined in view of Proposition~\ref{prop:Glinesearch}. Now, fix $n \in \N$. It follows from \eqref{eq:algoMT} that
\begin{equation}
    x_n - x_{n+1} -\gamma_{n} B x_n -\gamma_{n-1}v_n \in \gamma_{n} A x_{n+1}.
\end{equation}
Then, since $-\gamma_{n} B\widehat{x} \in \gamma_{n} A \widehat{x}$, the monotonicity of $A$ yields
\begin{align}\label{eq:proofprop11}
   0 \leq  \scal{x_n - x_{n+1}  -\gamma_{n} B x_n -\gamma_{n-1}v_n +\gamma_{n} B\widehat{x} }{x_{n+1} -\widehat x}.
\end{align}
In addition, 
\begin{equation}\label{eq:proofprop12}
     2\scal{x_n - x_{n+1} }{x_{n+1}-\widehat x}  = \|x_n-\widehat x\|^2 - \|x_n - x_{n+1}\|^2-\|x_{n+1}-\widehat x\|^2
\end{equation}
Then, by combining \eqref{eq:proofprop11} and \eqref{eq:proofprop12} we deduce
\begin{align}\label{eq:proofprop13}
    \|x_{n+1}&-\widehat x\|^2 +2\gamma_{n}  \scal{B x_n-B\widehat{x}}{x_{n+1}-\widehat x}\nonumber\\
    &\leq \|x_n-\widehat x\|^2 - 2\gamma_{n-1} \scal{v_n}{x_{n+1}-\widehat x}- \|x_n - x_{n+1}\|^2\nonumber\\
    &= \|x_n-\widehat x\|^2 - 2\gamma_{n-1} \scal{v_n}{x_n-\widehat x} - 2\gamma_{n-1} \scal{v_n}{x_{n+1}-x_n}- \|x_n - x_{n+1}\|^2.
\end{align}
Furthermore, by the monotonicity of $B$ we deduce 
\begin{align*}
    2\gamma_{n}  \scal{B x_n-B\widehat{x}}{x_{n+1}-\widehat x} &= 2\gamma_{n}  \scal{Bx_{n+1}-B\widehat{x}}{x_{n+1}-\widehat x}+2\gamma_{n}  \scal{B x_n-Bx_{n+1}}{x_{n+1}-\widehat x}\\
    &\geq 2\gamma_{n}  \scal{B x_n-Bx_{n+1}}{x_{n+1}-\widehat x}\\
    &=-2\gamma_{n}  \scal{v_{n+1}}{x_{n+1}-\widehat x}
\end{align*}
Hence, from \eqref{eq:proofprop13} we have 
\begin{align}\label{eq:proofprop14}
    \|x_{n+1}&-\widehat x\|^2 -2\gamma_{n}  \scal{v_{n+1}}{x_{n+1}-\widehat x}\nonumber\\
    &\leq \|x_n-\widehat x\|^2 - 2\gamma_{n-1} \scal{v_n}{x_n-\widehat x} - 2\gamma_{n-1} \scal{v_n}{x_{n+1}-x_n}- \|x_n - x_{n+1}\|^2.
\end{align}
Now, by applying Cauchy-Schwarz and Young's inequalities
\begin{align}\label{eq:proofprop15}
  -2\gamma_{n-1} \scal{v_n}{x_{n+1}-x_n} &\leq 2\gamma_{n-1}\|v_n\| \|x_{n+1}-x_n\|\nonumber\\
  &\leq \frac{\gamma_{n-1}^2}{\rho} \|v_n\|^2 + \rho \|x_{n+1}-x_n\|^2.  
\end{align}
Then, by combining \eqref{eq:proofprop14} and \eqref{eq:proofprop15}
\begin{align}\label{eq:proofprop16}
    \|x_{n+1}&-\widehat x\|^2 -2\gamma_{n}  \scal{v_{n+1}}{x_{n+1}-\widehat x}\nonumber\\
    &\leq \|x_n-\widehat x\|^2 - 2\gamma_{n-1} \scal{v_n}{x_n-\widehat x} + \frac{\gamma_{n-1}^2}{\rho} \|v_n\|^2-(1-\rho) \|x_n - x_{n+1}\|^2.
\end{align}
Moreover, by adding $\frac{\gamma_n^2}{\rho^2}\|v_{n+1}\|^2+\frac{\gamma_{n-1}^2}{\rho^2}\|v_{n}\|^2$ in \eqref{eq:proofprop16}, we deduce
\begin{align}\label{eq:proofprop17}
    &\|x_{n+1}-\widehat x\|^2 -2\gamma_{n}  \scal{v_{n+1}}{x_{n+1}-\widehat x}+\frac{\gamma_n^2}{\rho^2}\|v_{n+1}\|^2\nonumber\\
    &\leq \|x_n-\widehat x\|^2 - 2\gamma_{n-1} \scal{v_n}{x_n-\widehat x}+\frac{\gamma_{n-1}^2}{\rho^2}\|v_{n}\|^2- (1-\rho)\|x_n - x_{n+1}\|^2\nonumber\\
    &\hspace{5cm} +\frac{\gamma_n^2}{\rho^2}\|v_{n+1}\|^2-\left(\frac{\gamma_{n-1}^2}{\rho^2}-\frac{\gamma_{n-1}^2}{\rho} \right)\|v_n\|^2.
\end{align}
By \eqref{eq:linesearch} and Young's inequality
\begin{align}\label{eq:proofprop18}
 \frac{\gamma_n^2}{\rho^2}\|v_{n+1}\|^2 &\leq   \frac{\theta^2}{\rho^2}(\|x_{n+1}-x_n\|+\gamma_{n-1}\|v_n\|)^2\nonumber\\
 &\leq   \frac{\theta^2}{\rho^2}\left((1+\rho^2)\|x_{n+1}-x_n\|^2+\gamma_{n-1}^2\left(1+\frac{1}{\rho^2}\right)\|v_n\|^2\right).
\end{align}
Therefore, it follows from \eqref{eq:proofprop17} and \eqref{eq:proofprop18} that
\begin{align*}
    \Gamma_{n+1}(\widehat{x}) \leq \Gamma_{n}(\widehat{x}) -\left(1-\rho -\left(1+\frac{1}{\rho^2}\right)\theta^2\right)
     \left(\|x_n - x_{n+1}\|^2+\frac{\gamma_{n-1}^2}{\rho^2}\|v_{n}\|^2\right).
\end{align*}
The result follows.
\item Note that,
\begin{align*}
    \Gamma_n(\widehat{x}) &= \|x_{n}-\widehat{x}\|^2 - 2\gamma_{n-1}\scal{x_n-\widehat{x}}{v_n}+\frac{\gamma_{n-1}^2}
    {\rho^2} \|v_n\|^2\\
    &\geq \|x_{n}-\widehat{x}\|^2 - \rho^2\|x_n-\widehat{x}\|^2 - \frac{\gamma_{n-1}^2}{\rho^2} \|v_n\|^2 +\frac{\gamma_{n-1}^2}
    {\rho^2} \|v_n\|^2\\
    &= \left(1-\rho^2\right) \|x_n-\widehat{x}\|^2.
\end{align*}
\end{enumerate}
\end{proof}
\begin{rem}\label{rem:rhotheta}
In view of \eqref{eq:desGamma}, for guaranteeing that the sequence $(\Gamma_{n}(\hat{x}))_{n \in \N}$, defined in \eqref{eq:defGamma},
    is not increasing, we need
    \begin{equation}\label{eq:stepsizerem}
        1-\rho - \left(1+\frac{1}{\rho^2}\right)\theta^2 >0.
    \end{equation}
    The function $\rho \mapsto \rho + \left(1+\frac{1}{\rho^2}\right)\theta^2$ is minimized at $\rho = 2^{1/3} \theta^{2/3}$, then by replacing that value in  \eqref{eq:stepsizerem}, we deduce that $(\Gamma_{n}(\hat{x}))_{n \in \N}$ is not increasing if
    \begin{equation*}
        \theta^2+3\left(\frac{\theta}{2}\right)^{\frac{2}{3}}< 1.
    \end{equation*}
    \end{rem}
\begin{teo}\label{teo:convergencia} In the context of Problem~\ref{prob:problemMAIN}, let $(v_n,x_n)_{n \in \N}$ be generated by Algorithm~\ref{alg:algoMT}. Moreover, suppose that
\begin{equation} \label{eq:stepsize}
   \theta^2+3{\left(\frac{\theta}{2}\right)}^{\frac{2}{3}} < 1
\end{equation}
and that one of the following assertions holds:
\begin{enumerate}
    \item\label{teo:convergenciai} $\displaystyle \liminf_{n \to +\infty} \gamma_n = \gamma>0.$
    \item\label{teo:convergenciaii} $B$ is uniformly continuous in any weakly compact subset of $\overline{\dom} A$.
\end{enumerate}
Then $(x_n)_{n \in \N}$ converges weakly to a point in $\zer (A+B)$.
\end{teo}
\begin{proof} First, 
let $\widehat{x} \in \zer(A+B)$, set $\rho = 2^{1/3} \theta^{2/3}$, and let $(\Gamma_{n}(\widehat{x}))
_{n\in\N}$ be the sequence defined in \eqref{eq:defGamma}. In  view of \eqref{eq:desGamma}, \eqref{eq:stepsize}, Remark~\ref{rem:rhotheta}, and \cite[Lemma~5.31]{bauschkebook2017}, it follows that $(\Gamma_n(\widehat{x}))_{n \in \N}$ converges, $\sum_{n \in \N} \| x_{n+1} -x_{n} \|^{2} < +\infty$, and $\sum_{n \in \N} \gamma_n^2\| v_{n+1}\|^{2} < +\infty$. Then, we deduce that $(x_n)_{n \in \N}$ is bounded, $\|x_{n+1} - x_n\|\to 0$, and $\gamma_n\| v_{n+1}\| \to 0$. In addition, from the definition of $\Gamma_{n}(\widehat{x})
$ in \eqref{eq:defGamma}, we conclude that $(\|x_n-\widehat{x}\|)_{n \in  \N}$ also converges.
    \begin{enumerate} \item   Suppose now that $ \liminf_{n \to +\infty} \gamma_n = \gamma >0$. Since $\gamma_n\| v_{n+1}\| \to 0$ we deduce that  $\| v_{n+1}\| \to 0$. Now, let $\overline{x}$ be a weak cluster point of $(x_n)_{n \in \N}$, say $x_{n_k} \weak \overline{x}$. It follows from \eqref{eq:algoMT}  that
    \begin{equation}\label{eq:resolvent}
       (\forall k \in \N) \quad z_{n_k}:=\frac{x_{n_k} - x_{n_k+1}}{\gamma_{n_k}}+ v_{n_k+1} - \frac{\gamma_{n_k-1}}{\gamma_{n_k}}v_{n_k} \in (A+B)x_{n_k+1}
    \end{equation}
    Since $x_n-x_{n+1} \to 0$ and $v_n \to 0$, it follows  that $z_{n_k} \to 0$. Furthermore, since $A+B$ is maximally monotone, $\gra (A+B)$ is weak-strong closed \cite[Proposition~20.38]{bauschkebook2017}. Therefore, from \eqref{eq:resolvent}, we deduce that $0 \in ( A + B ) \overline{x}$. The result follows from \cite[Lemma~2.47]{bauschkebook2017}.
    \item Without loss of generality, suppose that $\liminf_{n \to +\infty} \gamma_n = 0$. Define, for each $n \in \N$, 
    \begin{equation}\label{eq:defwx}
        \widehat{\gamma}_n = \frac{\gamma_n}{\sigma} \quad \textnormal{ and } \quad \widehat{x}_{n+1}  = J_{\widehat{\gamma}_n  A} (x_n - \gamma_{n-1} v_n - \widehat{\gamma}_n Bx_n).
    \end{equation}
    Then, according Algorithm~\ref{alg:algoMT}, $\widehat{\gamma}_n$ was rejected in the linesearch step and we have 
    \begin{equation}\label{eq:gammarec}
        \|B\widehat{x}_{n+1} - Bx_{n}\| > \frac{\theta}{\widehat{\gamma}_n}( \|\widehat{x}_{n+1} - x_{n}\|+\gamma_{n-1}\|v_n\|).
    \end{equation}
    Now, since $\widehat{\gamma}_n  > \gamma_n$, the nonincreasing property of $\gamma \mapsto \frac{1}{\gamma} \|z- J_{\gamma A}(z-\gamma y)\|$ provided by Lemma~\ref{lemma} yields
    \begin{align*}
        \frac{\|x_n - \gamma_{n-1} v_n -\widehat{x}_{n+1}\|}{\widehat{\gamma}_n} \leq \frac{\|x_n - \gamma_{n-1}v_n  - x_{n+1}\|}{\gamma_n},
    \end{align*}
    which is equivalent to,
    \begin{align}\label{eq:noninc1}
        \sigma \|x_n - \gamma_{n-1} v_n - \widehat{x}_{n+1}\| \leq \|x_n - \gamma_{n-1} v_n - x_{n+1}\|.
    \end{align}
   By triangle inequality we have
   \begin{equation}\label{eq:noninc2}
     \sigma\|x_n-\widehat{x}_{n+1}\|-\sigma\gamma_{n-1} \|v_n\| \leq \sigma\| x_n - \gamma_{n-1} v_n - \widehat{x}_{n+1}\|
   \end{equation}
   and
      \begin{equation}\label{eq:noninc3}
     \|x_n - \gamma_{n-1} v_n - x_{n+1}\| \leq \|x_n - x_{n+1}\| + \gamma_{n-1}\| v_n \|.
   \end{equation}
   Then, \eqref{eq:noninc1}, \eqref{eq:noninc2},  and \eqref{eq:noninc3} imply that
   \begin{align*}
      \sigma \|x_n-\widehat{x}_{n+1}\| &\leq \|x_n-x_{n+1}\|+(1+\sigma)\gamma_{n-1}\|v_n\|,
   \end{align*}
   which yields
   \begin{equation}\label{eq:xwxto0}
       x_n-\widehat{x}_{n+1} \to 0.
   \end{equation}
   Now let $\overline{x}$ be a weak cluster point of $(x_n)_{n \in \N}$ and let $(x_{n_k})_{k \in \N}$ be a subsequence such that $x_{n_k} \weak \overline{x}$. It follows from \eqref{eq:xwxto0} that $\widehat{x}_{n_k}\weak \overline{x}$. Since $\{\overline{x}\} \bigcup_{k \in \N}[x_{n_k},\widehat{x}_{n_k+1}]$ is a weakly compact subset of $\overline{\conv} \dom A = \overline{\dom} A $ \cite[Lemma~3.2]{salzo2017}, it follows from the uniform continuity of $B$ and \eqref{eq:xwxto0} that 
   \begin{equation}\label{eq:BxBwxto0}
       B \widehat{x}_{n_k+1} - B x_{n_k} \to 0.
   \end{equation}
    Moreover, from \eqref{eq:gammarec}, we deduce that
       \begin{equation}\label{eq:xwxto02}
       \frac{\widehat{x}_{n_k+1} -  x_{n_k}}{\widehat{\gamma}_{n_k}}\ \to 0 \quad \textnormal{ and } \quad \frac{\gamma_{n_k-1}}{\widehat{\gamma}_{n_k}}v_{n_k} \to 0.
   \end{equation}
   It follows from \eqref{eq:defwx}  that
    \begin{equation}\label{eq:resolvent2}
       (\forall k \in \N) \quad \widehat{z}_{n_k+1} :=\frac{x_{n_k} - \widehat{x}_{n_k+1}}{\widehat\gamma_{n_k}}+ B\widehat x_{n_k+1} - Bx_{n_k} - \frac{\gamma_{n_k-1}}{\widehat\gamma_{n_k}}v_{n_k}  \in (A+B)\widehat x_{n_k+1}.
    \end{equation}
    Finally, in view of \eqref{eq:xwxto02} and \eqref{eq:BxBwxto0}, $\widehat{z}_{n_k+1} \to 0$ and the result follows similarly to case \ref{teo:convergenciai}. 
\end{enumerate}
\end{proof}
\begin{rem}
Algorithm~\ref{alg:algoMT} considers, at each iteration, $\alpha_0$ as the initial step-size in the linesearch subroutine. While this allows for the possibility of taking larger steps, in some instances it can slow down the algorithm because the number of inner linesearch subiterations can be significantly large. To avoid this issue, it is possible not to restart the step-size parameter in the linesearch and instead initialize it with $\gamma_{n-1}$. The convergence proof for this variation of Algorithm~\ref{alg:algoMT} is analogous. In Section~\ref{sec:num}, we present numerical experiments comparing the performance of the algorithm with and without restarting the step-size parameter.
\end{rem}
Next, we revisit Example~\ref{example1} to demonstrate that the linesearch procedure in Algorithm~\ref{alg:algoMT} is well defined for this scenario.
\begin{example} Consider the same setting as in Example~\ref{example1} we have $\gamma_0=1$, $v_1=(0,1)$, $\|x_{\gamma(t)} - x_1\| = t^2 \sqrt{1 + t^2}$, and $\|Bx_{\gamma(t)} - Bx_1\| = t^{2/3} \sqrt{1 + t^{14/3}}$ with $\gamma = \frac{3}{2}t^4 + t^2 + \frac{3}{2}t$. Then, the linesearch condition in \eqref{eq:linesearch} reduces to:
\begin{equation}
    \left( \frac{3}{2}t^4 + t^2 + \frac{3}{2}t \right) t^{2/3} \sqrt{1 + t^{14/3}} \leq \theta \left(t^2 \sqrt{1 + t^2}+ 1 \cdot \|(1,0)\|\right) = \theta t^2 \sqrt{1 + t^2}+ \theta,
\end{equation}
which clearly holds for sufficiently small $t \in \RPP$.
\end{example}
    \section{Four Operator Splitting  }\label{sec:cocoandlips}
In this section, we consider the following monotone inclusion 
problem.
\begin{pro}\label{prob:problem1}
Let $A: \H \to 2^\H$ be a  maximally monotone operator, let 
$C : \H \to \H$ be a $\beta$-cocoercive operator for some 
$\beta 
>0$, let $D :\H \to \H$ be a monotone and $L$-Lipschitz continuous 
operator for some $L>0$, and let $B : \H \to 2^\H$ be a 
maximally 
monotone operator such that $B$ is single valued and 
continuous 
in $\overline{\dom}A \subset \dom B$. Moreover assume that 
$A+B$ 
is 
maximally 
monotone. The 
problem is to 
\begin{equation}\label{eq:problem1E}
\text{find} \quad x \in \H \quad \text{such that} \quad 0 \in 
Ax+Bx+Cx+Dx,
\end{equation}
under the assumption that the set of solutions to 
\eqref{prob:problem1} is nonempty.
\end{pro}	
Note that $B+C+D$ is a continuous operator; thus, this problem can be viewed as a particular instance of Problem~\ref{prob:problemMAIN}. However, as shown by the authors in \cite{BricenoRoldan4op}, separating the operators to exploit their individual properties can significantly improve the numerical performance and convergence of the algorithm. The following algorithm is an extension of Algorithm~\ref{alg:algoMT} for solving Problem~\ref{prob:problem1}.
\begin{algo}\label{alg:algoMTE}
		In the context of Problem~\ref{prob:problem1}, let $\gamma_{-1} \in \RPP$, let $(\theta,\sigma) \in ~ ]0,1[^2$, let $\tau \in \RPP$, let $(\alpha_k)_{k \in \N}$ be defined by $\alpha_k =\tau \sigma^k$, and let $(x_0,x_{-1}) \in (\dom A)^2$. Define $v_0 = Bx_0-Bx_{-1}+ Dx_0-Dx_{-1}$ and consider the sequence defined recursively by
		\begin{equation}
			\label{eq:algoMTE}
			(\forall n \in \N)\quad \left\lfloor
			\begin{aligned}
                & v_{n}  = Bx_{n} - B x_{n-1} +Dx_{n} - Dx_{n-1}, \\
                &x_{n+1} = J_{\gamma_n A} \left( x_n -\gamma_n (B+C+D) x_n -\gamma_{n-1}v_n\right),
			\end{aligned}
			\right.
		\end{equation} 
        where, for each $n \in \N$, $\gamma_n = \alpha_k$ where $k$ is the smallest natural number such that 
        \begin{equation}\label{eq:linesearchE}
            \gamma_{n}\|v_{n+1}\| \leq (\theta + \gamma_n L) \|x_{n+1}-x_n\| + \theta \gamma_{n-1}\|v_n\|.
        \end{equation}
	\end{algo}	
The following proposition ensures that Algorithm~\ref{alg:algoMTE} is well defined and it establishes key estimates necessary to derive the convergence of the algorithm.
\begin{prop}\label{prop:MTFejer22}
    In the context of Problem~\ref{prob:problem1} let $(v_n,x_n)_{n \in \N}$ be generated by Algorithm~\ref{alg:algoMTE}. In addition, let $\widehat x \in \zer(A+B+C+D)$, and define, for each $n \in \N$,
    \begin{equation}\label{eq:defGamma2}
        \Gamma_n(\widehat{x}) = \|x_{n}-\widehat{x}\|^2 - 2\gamma_{n-1}\scal{x_n-\widehat{x}}{v_n}+\frac{\gamma_{n-1}^2}
        {\rho^2} \|v_n\|^2.
    \end{equation}
    Then, the following assertions hold:
    \begin{enumerate}
    \item\label{prop:MTFejer220}  For each $n \in \N$, there exists $k_n \in \N$ such that \eqref{eq:linesearchE} holds for $\gamma_n = \alpha_{k_n}$.
        \item\label{prop:MTFejer221} Let $\varepsilon \in \RPP$ and suppose that $2\beta\varepsilon \geq \tau$. Then, for every $n \in \N$,
            \begin{align}\label{eq:desGamma2}
             \Gamma_{n+1}(\widehat{x}) \leq \Gamma_{n}(\widehat{x}) &-\left(1-\kappa^2  -\frac{(\theta+L)}{\rho^2}\left(\theta(1+\rho^2)+L\gamma_n^2\right) - \frac{\gamma_n}{2\beta} \right)
            \|x_{n+1}-x_n\|^2 \nonumber\\
            &\qquad - \left( 1 -\frac{\rho^2}{\kappa^2}-\theta\left(\theta+L\right)\left(1+\frac{1}{\rho^2}\right)\right)\frac{\gamma_{n-1}^2}{\rho^2}\|v_{n}\|^2  
            \end{align}
        \item \label{prop:MTFejer222} For every $n \in \N$
        \begin{equation}\label{eq:desGammabel2} \Gamma_{n}(\widehat x) \geq \left(1-\rho^2\right) \|x_n-\widehat{x}\|^2.
        \end{equation}
    \end{enumerate}
\end{prop}
\begin{proof} \begin{enumerate}
\item  Fix $n \in \N$. It follows from Proposition~\ref{prop:Glinesearch} that there exists $k_n\in \N$ such that, for $\gamma_n= \alpha_{k_n}$, we have
\begin{equation*}
    \gamma_n\|Bx_{n+1}-Bx_n\| \leq \theta(\|x_{n+1}-x_n\|+\gamma_{n-1}\|v_n\|).
\end{equation*}
Hence, by the Lipschitz continuous property of $D$ we deduce
\begin{equation*}
    \gamma_n \|v_{n+1}\| \leq \gamma_n \|Bx_{n+1}-Bx_n\|+\gamma_n\|Dx_{n+1}-Dx_n\|\leq \theta(\|x_{n+1}-x_n\|+\gamma_{n-1}\|v_n\|) +\gamma_n L \|x_{n+1}-x_n\|,
\end{equation*}
which yields the result.
\item     Fix $n \in \N$. It follows from \eqref{eq:algoMTE} that
    \begin{align*}
        x_{n} - x_{n+1} - \gamma_n (B+C+D)x_{n} - \gamma_{n-1}v_n \in \gamma_n A x_{n+1}
    \end{align*}
    Then, since $-\gamma_n (B+C+D) \hat{x} \in \gamma_n A \hat{x}$, the monotonicity of $A$ yields
    \begin{equation}\label{eq:prooffejer21}
        0 \leq \scal{x_n - x_{n+1} - \gamma_n (B+C+D)x_n - \gamma_{n-1} v_n + \gamma_n (B+C+D) \hat{x} }{x_{n+1} - \hat{x}}.
    \end{equation}
  Noting that $B+D$ is monotone and proceeding similarly to the proof of Proposition~\ref{prop:MTFejer}, we deduce that 
\begin{align}\label{eq:prooffejer22}
    \|x_{n+1}&-\widehat x\|^2 -2\gamma_{n}  \scal{v_{n+1}}{x_{n+1}-\widehat x}\nonumber\\
    &\leq \|x_n-\widehat x\|^2 - 2\gamma_{n-1} \scal{v_n}{x_n-\widehat x} - 2\gamma_{n-1} \scal{v_n}{x_{n+1}-x_n}\nonumber\\
    & \hspace{5cm}- \|x_n - x_{n+1}\|^2-2\gamma_n \scal{Cx_{n} - C \hat{x}}{x_{n+1} - \hat{x}}.
\end{align}
    By applying Cauchy-Schwarz and Young's inequalities, for every $\kappa \in \RPP$, we have
    \begin{align}\label{eq:prooffejer23}
        2 \gamma_{n-1}\scal{v_n}{x_{n+1} - x_{n}} &\leq 2 \gamma_{n-1}\|v_n\| \|x_{n+1} - x_{n} \| \nonumber\\
        &\leq \frac{\gamma_{n-1}^2}{\kappa^2} \|v_n\|^2 +\kappa^2\|x_{n+1} - x_{n}\|^2.
    \end{align}
    Moreover, by the $\beta$-cocoercivity of $C$,
    \begin{align}\label{eq:prooffejer24}
        - 2\gamma_n\scal{Cx_n - C\hat{x}}{x_{n+1} - \hat{x}} &=- 2\gamma_n\scal{Cx_n - C\hat{x}}{x_{n+1} - x_n}- 2\gamma_n\scal{Cx_n - C\hat{x}}{x_{n} - \hat{x}}\nonumber\\
        &\leq 2\gamma_{n}\|C x_n - C \hat{x}\| \|x_{n+1} - x_n\| - 2\gamma_n \beta\|Cx_n - C\hat{x}\|^2 \nonumber\\
        &\leq \frac{\gamma_n ^2}{\varepsilon} \|C x_n - C \hat{x}\|^2 + \varepsilon\|x_{n+1} - x_n\|^2 - 2\gamma_n \beta \|Cx_n - C\hat{x}\|^2 \nonumber\\
        &= \varepsilon \|x_{n+1} - x_n\|^2 - \frac{\gamma_n}{\varepsilon}\left(2\beta\varepsilon - \gamma_n \right)\|C x_n - C \hat{x}\|^2.
    \end{align}
   Combining \eqref{eq:prooffejer22}-\eqref{eq:prooffejer24} and adding $\dfrac{\gamma_{n}^2}{\rho^2}\|v_{n+1}\|^2 + \dfrac{\gamma_{n-1}^2}{\rho^2}\|v_{n}\|^2$ we obtain
        \begin{align}\label{eq:prooffejer25}
          \Gamma_          {n+1}(\hat{x})&\leq \Gamma_
          {n}(\hat{x})  -(1-\kappa^2-\varepsilon) \|x_{n+1}-x_n\|^2- \left( \frac{\gamma_{n-1}^2}{\rho^2}  - \frac{\gamma_{n-1}^2}{\kappa^2} \right) \|v_n\|^2\nonumber \\
          & \hspace{2cm}- \frac{\gamma_n}{\varepsilon}\left(2\beta\varepsilon - \gamma_n \right)\|C x_n - C \hat{x}\|^2+\dfrac{\gamma_{n}^2}{\rho^2}\|v_{n+1}\|^2 \nonumber \\
          &\leq \Gamma_
          {n}(\hat{x})  -(1-\kappa^2-\varepsilon) \|x_{n+1}-x_n\|^2- \left( \frac{\gamma_{n-1}^2}{\rho^2}  - \frac{\gamma_{n-1}^2}{\kappa^2} \right) \|v_n\|^2\nonumber +\dfrac{\gamma_{n}^2}{\rho^2}\|v_{n+1}\|^2 ,\nonumber
    \end{align}
    where the last inequality follows from $2\beta\varepsilon \geq \tau \geq \gamma_n$.  Finally, the result follows noticing that, in view of \eqref{eq:linesearchE}, we have
    \begin{align*}
        \frac{\gamma_n^2}{\rho^2}\|v_{n+1}\|^2 
        &\leq \frac{1}{\rho^2}((\theta+\gamma_nL)\|x_{n+1}-x_n\|+\theta\gamma_{n-1}\|v_n\|)^2\\
        & \leq  \frac{\theta^2}{\rho^2}\left(1+\frac{L}{\theta}\right)(\|x_{n+1}-x_n\|+\gamma_{n-1}\|v_n\|)^2+\frac{\gamma_n^2L^2}{\rho^2}\left(1+\frac{\theta}{L}\right)\|x_{n+1}-x_n\|^2 \\
        & \leq  \frac{\theta^2}{\rho^2}\left(1+\frac{L}{\theta}\right)\left((1+\rho^2)\|x_{n+1}-x_n\|^2+\left(1+\frac{1}{\rho^2}\right)\gamma_{n-1}^2\|v_n\|^2\right)\\
                &\hspace{3cm}+\frac{\gamma_n^2L^2}{\rho^2}\left(1+\frac{\theta}{L}\right)\|x_{n+1}-x_n\|^2 \\
        & =  \frac{(\theta+L)}{\rho^2}\left(\theta(1+\rho^2)+L\gamma_n^2\right)\|x_{n+1}-x_n\|^2+\frac{\theta}{\rho^2}\left(\theta+L\right)\left(1+\frac{1}{\rho^2}\right)\gamma_{n-1}^2\|v_n\|^2.
    \end{align*}
    \item This is directly by using Cauchy-Schwarz and Young's inequalities.
\end{enumerate}
\end{proof}

\begin{teo}\label{teo:convergencia2}
    In the context of Problem~\ref{prob:problem1} let $(v_n,x_n)_{n \in \N}$ be generated by Algorithm~\ref{alg:algoMTE}. Let $\varepsilon \in \RPP$ be such that $2\beta\varepsilon \geq \tau$ and suppose that
    \begin{equation} \label{eq:stepsize2}
       1-\kappa^2  -\frac{(\theta+L)}{\rho^2}\left(\theta(1+\rho^2)+L\tau^2\right) - \varepsilon > 0 \quad \textnormal{ and } \quad 1 -\frac{\rho^2}{\kappa^2}-\theta\left(\theta+L\right)\left(1+\frac{1}{\rho^2}\right)> 0.
    \end{equation}
     Moreover, assume that one of the following assertions holds
    \begin{enumerate}
        \item\label{teo:convergencia21} $\displaystyle \liminf_{n \to +\infty} \gamma_n = \gamma>0.$
        \item\label{teo:convergencia22} $B$ is uniformly continuous in any weakly compact subset of $\overline{\dom} A$.
    \end{enumerate}
    Then $(x_n)_{n \in \N}$ converges weakly to a point in $\zer (A+B+C+D)$.
\end{teo}
\begin{proof} The condition \eqref{eq:stepsize2} guarantees that $(\Gamma_n(\widehat{x}))_{n \in \N}$, defined in \eqref{eq:defGamma2}, is nonnegative for every $\widehat{x} \in \zer (A+B+C+D)$. The proof is analogous to the proof of Theorem~\ref{teo:convergencia}.
\end{proof}
\begin{rem} \label{remthetarho}
\begin{enumerate}
    \item Note that if $\kappa > \rho$ \eqref{eq:stepsize2} holds for $\theta$ and $\tau$ sufficiently small. 
    \item Suppose that $\kappa^2 \geq \rho$. Then,
    \begin{align*}
         1-\kappa^2  -\frac{(\theta+L)}{\rho^2}\left(\theta(1+\rho^2)+L\tau^2\right) - \varepsilon &  <  1-\kappa^2  -\frac{(\theta+L)}{\rho^2}\left(\theta(1+\rho^2)\right) \\
         &\leq  1-\frac{\rho^2}{\kappa^2}  -\frac{(\theta+L)}{\rho^2}\left(\theta(1+\rho^2)\right).
    \end{align*}
    Hence, the second inequality in \eqref{eq:stepsize2} follows directly from the first one.
    \item Let $\alpha = \frac{(\theta+L)}{\rho^2}\theta(1+\rho^2)$. The condition \eqref{eq:stepsize2} implies
    \begin{equation*}
        1-\alpha > \kappa^2 > \frac{\rho^2}{1-\alpha}.
    \end{equation*}
    Therefore, a necessary condition for \eqref{eq:stepsize2} to hold is that $1-\alpha >\frac{\rho^2}{1-\alpha}$ which is equivalent to
    \begin{equation}\label{eq:rhotheta}
        \psi(\rho):=1-\rho-\frac{(\theta+L)}{\rho^2}\theta(1+\rho^2) >0.
    \end{equation}
    The function $\psi(\rho)$ is attains its maximum at $\rho = (2(\theta+L)\theta)^{1/3}$. This value of $\rho$ can be used to simplify the parameters selection.
    \item When $B=0$, by setting $\theta = 0$ and $\rho = \kappa = \sqrt{\tau L}$, \eqref{eq:stepsize2} reduces to the condition guaranteeing the convergence of FHRB, as presented in \cite[Theorem~5.2]{Malitsky2020SIAMJO}.
\end{enumerate}
\end{rem}
In the following example we apply Algorithm~\ref{alg:algoMTE} for solving convex optimization problems with nonlinear constraints. This problem was studied in \cite[Section~4]{BricenoRoldan4op}.
\begin{example}\label{ex:nonlinearconstraints}
 		Let $f\in\Gamma_0 (\H)$, let $g\in \Gamma_0(\G)$, let $h:\H \to 
	\R$ 
	be a convex G\^ateaux differentiable
	function such that $\nabla h$ is $\beta^{-1}$-Lipschitz continuous for 
	some $\beta \in \RPP$, let $M \colon \H \to \G$ be a bounded 
	linear 
	operator, and let $e \colon \H \to \RX^p \colon x 
	\mapsto 
	(e_i(x))_{1\le i 
		\le p}$ be such that, for every $i\in \{1,\ldots,p\}$, $e_i$ 
	is 
	convex and 
	G\^ateaux differentiable in $\inte \dom e_i$, $\dom e_i$ is 
	closed, 
	$\cap_{i=1}^p 
	\interior \dom 
	e_{i}\neq \varnothing$, and $\dom \partial f \subset  
	\cap_{i=1}^n \interior \dom  e_i $. Assume that $0 \in 
	\sri(\dom g - M (\dom f))$ and that 
	\begin{equation}\label{e:slater}
		\begin{cases}
			(\forall i \in \{1,\ldots,p\}) \quad \quad \lev_{\leq 0} 
			e_i \subset 
			\inte \dom e_i;\\
			\dom (f + g \circ M) \cap \bigcap_{i=1}^n \lev_{<0} e_i 
			\neq \varnothing.
		\end{cases}
	\end{equation} The problem is to 
	\begin{equation}\label{eq:probopti}
		\min_{e(x)\in \RM^p} f(x) + g (Mx) + h(x),
	\end{equation}
	and we assume that solutions exist.

    Let $\HH = \H \times \G \times \R^p$ and define the operators
		\begin{equation}\label{eq:defoperators}
			\begin{cases}
				A \colon \HH\to 2^{\HH} \colon (x,u,v) \mapsto 
				\partial f (x) \times 
				\partial g^*(u) \times N_{\RP^p}(v),\\
				C \colon \HH\to \HH \colon (x,u,v) \mapsto (\nabla 
				h (x) , 0 , 0),\\
				D \colon \HH\to \HH \colon (x,u,v) \mapsto 
				(M^*u,-Mx,0),\\
                	B \colon \HH\to 2^{\HH} 
			\nonumber\\
			(x,u,v) \mapsto \begin{cases}
				\!\!\left\{\!\left(\displaystyle{\sum_{i=1}^p} v_i 
				\nabla  e_i (x) , 
				0, 
				-e(x)\right)\!\!\right\},\hspace{-.3cm}& \text{ if } 
				v \in 
				\RP^p \text{ and } x \in \bigcap_{i=1}^p \interior 
				\dom 
				e_{i};\\
				\varnothing,& \text{ otherwise.} 		
			\end{cases}
			\end{cases}
		\end{equation}
We have that $A$ is maximally monotone, $C$ is $\beta$-cocoercive, and $D$ is $\|M\|$-Lipschitz continuous. Moreover, in view of \cite[Proposition~4.2]{BricenoRoldan4op}, $B$ is maximally monotone and it is uniformly continuous in every compact subset of $\overline{\dom} \partial f \times \overline{\dom} \partial 
			g^* \times 
			\RP^p$ if one of the following assertions holds:
			\begin{enumerate}
				\item\label{prop:hip1} $(\nabla e_i)_{1\leq i \leq 
					p}$ are 
				bounded and uniformly 
				continuous in every weakly compact subset of 
				$\overline{\dom} \partial f$. 
				\item\label{prop:hip2} $\H$ is finite dimensional 
				and  
				$(\nabla e_i)_{1\leq i \leq p}$ are 
				continuous in every compact subset of 
				$\overline{\dom} \partial 
				f$.
			\end{enumerate}
			Therefore, the optimization problem in \eqref{eq:probopti} is a particular instance of Problem~\ref{prob:problem1} for the operators defined in \eqref{eq:defoperators}. Note that, since $\dom (\partial f) 
		\subset  
		\cap_{i=1}^n \interior \dom  e_i $ we have $\dom A \subset \dom B$. Hence, the optimization problem can be solved by the Algorithm~\ref{alg:algoMTE}. In particular, let $(x_0^1,x_0^2,x_0^3)\in \H\times \G 
		\times 
		\R^p$,  let $(x_{-1}^1,x_{-1}^2,x_{-1}^3)\in \H\times \G 
		\times 
		\R^p$, let  $(\theta,\sigma,\kappa,\rho) \in ~ ]0,1[^4$ be such that
        \begin{equation}
       1-\kappa^2  -\frac{(\theta+\|M\|)}{\rho^2}\left(\theta(1+\rho^2)+\|M\|\tau^2\right) - \varepsilon > 0 \quad \textnormal{ and } \quad 1 -\frac{\rho^2}{\kappa^2}-\theta\left(\theta+\|M\|\right)\left(1+\frac{1}{\rho^2}\right)> 0,
    \end{equation} 
   where $\varepsilon \in \RPP$ satisfies $2\beta \varepsilon \geq \tau$. In this context, Algorithm~\ref{alg:algoMTE} iterates 
		\begin{equation}\label{eq:algorithmop}
			(n\in\N)\quad\begin{array}{l}
				\left\lfloor
				\begin{array}{l}
					v_n^1 = 
					\sum_{i=1}^p x_{n,i}^3 \nabla e_i(x_n^1) 
					 - \sum_{i=1}^p x_{n-1,i}^3 \nabla e_i(x_{n-1}^1)
					 + M^*x_n^2 - M^*x_{n-1}^2  \\
                    v_n^2 =  Mx_{n-1}^1-Mx_n^1\\
                    v_n^3 = e(x_{n-1}^1) -e(x_{n}^1)\\
                    x_{n+1}^1 =  \prox_{\gamma_n 
				f}\left(x_n^1-\gamma_n \left( \nabla h (x_n^1) + M^*x_n^2 + 
			\sum_{i=1}^p 
			x_{n,i}^3 \nabla e_i (x_n^1)\right)- \gamma_{n-1} v_n^1\right) \\
                    x_{n+1}^2 = \prox_{\gamma_n
				g^*} 
			(x_n^2+\gamma_n 
			M x_n^1 - \gamma_{n-1}v_n^2)  \\
                    x_{n+1}^3 = P_{\RP^p}\big(x_n^3+\gamma_n e 
			(x_n^1)-\gamma_{n-1} v_n^3\big) ,
				\end{array}
				\right.
			\end{array}
		\end{equation}        
        where, for each $n \in \N$, $\gamma_n$ is the largest value in $\{\tau, \tau \sigma, \tau \sigma^2, \ldots\}$, such that
        \begin{equation*}
            \gamma_n\|(v_{n+1}^1,v_{n+1}^2,v_{n+1}^3)\| \leq (\theta+\gamma_n\|M \|) \|(x_{n+1}^1, x_{n+1}^2, x_{n+1}^3)-(x_n^1, x_n^2, x_n^3)\|+\theta \gamma_{n-1}\|(v_{n}^1,v_{n}^2,v_{n}^3)\|.
        \end{equation*}
        For additional technical details regarding this example, see \cite[Section~4]{BricenoRoldan4op}.
\end{example}

\section{Numerical Experiments}\label{sec:num}
In this section we present two numerical examples to test the numerical performance of the proposed methods. In particular, we compare FRB with and without linesearch in saddle point problems. Next, we compare FHRB and FBHF with linesearch in the context of image deblurring problem with $\ell^p$ regularization terms for $p \in~]1,2[$. All numerical experiments were implemented in MATLAB on a desktop computer equipped with an Intel Core i7-14700K processor (3.4/5.6~GHz), 64~GB of RAM, and running Windows~11 Pro 64-bit. The code is available in this  \href{https://drive.google.com/file/d/1oANDJOOHLyJTHgXGNHk3eIyoDUbd2RHf/view?usp=drive_link}{repository}.

\subsection{Saddle Point Problems} 
Consider the following minimax problem:

\begin{pro}\label{prob:saddle}
	Let $f \in \Gamma_0(\H)$, $g \in \Gamma_0(\G)$, and let $L \colon \H\to \G$ be a bounded linear operator. The problem consists of solving
	\begin{equation}\label{eq:problemsaddle}
		\min_{x\in \H}\max_{y \in \G} f(x)+ \scal{Lx}{y}-g(y),
	\end{equation}
	assuming its solution set is nonempty.
\end{pro}

This saddle point problems arise in several applications including zero-sum games \cite{vonNeumannMorgenstern1944}, robust optimization \cite{RobustOpti2009}, generalized lasso problems \cite{Tibshirani1996,Tibshirani2011}, and generative adversarial networks \cite{gidel2019variational,Mescheder2017}, among others. By defining $A \colon \H \times \G \to 2^{\H \times \G} \colon (x,y)\mapsto \partial f(x) \times \partial g(y)$ and $B \colon \H \times \G \to {\H \times \G} \colon (x,y)\mapsto (L^*y,-Lx)$, this problem can be equivalently formulated as finding $(x,y) \in \H\times \G$ such that 
\begin{equation}
	(0,0) \in A(x,y) + B(x,y).
\end{equation}

Furthermore, since $A$ is a maximally monotone operator and $B$ is continuous with full domain, Problem~\ref{prob:saddle} is a particular instance of Problem~\ref{prob:problemMAIN} and can therefore be solved using Algorithm~\ref{alg:algoMT}. Additionally, noting that $B$ is $\|L\|$-Lipschitz continuous, the problem can also be solved by the standard FRB method without a linesearch. 

We compare the numerical performance of FRB without a linesearch against three variants: FRB with the linesearch proposed in \cite{Malitsky2020SIAMJO} for locally Lipschitz operators (FRBLSL), Algorithm~\ref{alg:algoMT} without restarting the step-size parameter at each iteration (FRBLS), and Algorithm~\ref{alg:algoMT} with step-size restarting (FRBLSR). For this comparison, we consider the following setting: $\H =\R^N$, $\G =\R^M$, $f (x) = \frac{1}{2}x^\top Q x + q^\top x$, and $g(y) = \iota_{[-1,1]^M}(y)$, where $N, M \in \N$, $L \in \R^{M\times N}$, $Q \in \mathbb{R}^{N \times N}$ is a symmetric positive definite matrix, and $q \in \R^N$. 

We evaluate the algorithms across 16 pairs of dimensions $(N,M)$, which are detailed in Table~\ref{T:results}. For each pair $(N,M)$, we randomly generate 20 instances of the matrices $Q$, $q$, and $L$ using MATLAB's \texttt{randn} function. As a stopping criterion, we use a relative error tolerance of $10^{-6}$ alongside a maximum limit of $10^6$ iterations. FRB was implemented with $\gamma = 0.99/(2\|L\|)$; FRBLSL with $\theta = 0.99$, $\sigma = 0.8$, and $\gamma_0=\gamma/\sigma^3$; and both FRBLS and FRBLSR with $\theta = 0.3525$, $\tau = 0.8$, and $\gamma_0 = \gamma/\tau^3$. 

The numerical results, expressed in terms of the average number of iterations and average CPU time for each dimension, are presented in Table~\ref{T:results}. From this table, we observe that in all cases, the algorithms utilizing the proposed linesearch outperform the standard FRB. In particular, for a fixed $N$, FRBLS achieves the lowest CPU time when $M$ is relatively small. Note that in these specific cases, although FRBLSR requires fewer iterations to converge, it consumes more CPU time. This behavior is explained by the fact that, at each iteration, the number of inner linesearch evaluations FRBLSR needs to find an admissible step-size is higher compared to FRBLSL and FRBLS, which rely on the step-size from the previous step. This is attributable to the fact that when $M$ is small, the dimension of the null space of $L$ is larger, making the linesearch condition more restrictive. Conversely, in the remaining scenarios (larger $M$), FRBLSR yields the best performance. This demonstrates that, even when dealing with Lipschitz continuous operators, incorporating an adaptive linesearch can significantly accelerate the FRB algorithm.

\begin{table}[h!]
	\centering
	\begin{tabular}{l cc cc cc cc}
		\toprule
		$\bm{N = 500}$ & \multicolumn{2}{c}{$\bm{M = 150}$} 
		& \multicolumn{2}{c}{$\bm{M = 250}$} 
		& \multicolumn{2}{c}{$\bm{M = 400}$} 
		& \multicolumn{2}{c}{$\bm{M = 450}$} \\
		\cmidrule(lr){2-3} \cmidrule(lr){4-5} \cmidrule(lr){6-7} \cmidrule(lr){8-9}
		\textbf{Algorithm} & \textbf{NI} & \textbf{T} 
		& \textbf{NI} & \textbf{T} 
		& \textbf{NI} & \textbf{T} 
		& \textbf{NI} & \textbf{T} \\
		\midrule
		FRB    & 492 & 0.83 & 1063 & 1.81 & 2706 & 5.29 & 3432 & 6.76 \\
		FRBLSL & 401 & 0.68 & 871  & 1.49 & 2226 & 4.37 & 2823 & 5.58 \\
		FRBLS & 326 & \textbf{0.55} & 714  & 1.22 & 1829 & 3.59 & 2325 & 4.59 \\
		FRBLSR & 305 & 0.84 & 592  & \textbf{1.06} & 1511 & \textbf{3.03} & 1920 & \textbf{3.87} \\
		
		\midrule
		$\bm{N = 1000}$ & \multicolumn{2}{c}{$\bm{M = 300}$} 
		& \multicolumn{2}{c}{$\bm{M = 500}$} 
		& \multicolumn{2}{c}{$\bm{M = 750}$} 
		& \multicolumn{2}{c}{$\bm{M = 900}$} \\
		\cmidrule(lr){2-3} \cmidrule(lr){4-5} \cmidrule(lr){6-7} \cmidrule(lr){8-9}
		\textbf{Algorithm} & \textbf{NI} & \textbf{T} 
		& \textbf{NI} & \textbf{T} 
		& \textbf{NI} & \textbf{T} 
		& \textbf{NI} & \textbf{T} \\
		\midrule
		FRB    & 699 & 4.26 & 1505 & 9.48 & 3231 & 21.00 & 4828 & 31.95 \\
		FRBLSL & 572 & 3.50 & 1231 & 7.76 & 2652 & 17.27 & 3973 & 26.31 \\
		FRBLS & 448 & \textbf{2.73} & 1009 & 6.36 & 2179 & 14.18 & 3267 & 21.65 \\
		FRBLSR & 419 & 3.73 & 840  & \textbf{5.67} & 1800 & \textbf{12.05} & 2695 & \textbf{18.19} \\
		
		\midrule
		$\bm{N = 2000}$ & \multicolumn{2}{c}{$\bm{M = 600}$} 
		& \multicolumn{2}{c}{$\bm{M = 1000}$} 
		& \multicolumn{2}{c}{$\bm{M = 1600}$} 
		& \multicolumn{2}{c}{$\bm{M = 1800}$} \\
		\cmidrule(lr){2-3} \cmidrule(lr){4-5} \cmidrule(lr){6-7} \cmidrule(lr){8-9}
		\textbf{Algorithm} & \textbf{NI} & \textbf{T} 
		& \textbf{NI} & \textbf{T} 
		& \textbf{NI} & \textbf{T} 
		& \textbf{NI} & \textbf{T} \\
		\midrule
		FRB    & 1001 & 24.24 & 2176 & 54.21 & 5391 & 139.31 & 7130 & 186.30 \\
		FRBLSL & 817  & 19.86 & 1782 & 44.48 & 4431 & 114.52 & 5869 & 153.51 \\
		FRBLS & 602  & \textbf{14.59} & 1460 & 36.35 & 3640 & 94.02  & 4827 & 126.19 \\
		FRBLSR & 566  & 16.06 & 1209 & \textbf{31.59} & 3003 & \textbf{79.44}  & 3983 & \textbf{106.01} \\
		
		\midrule
		$\bm{N = 3000}$ & \multicolumn{2}{c}{$\bm{M = 900}$} 
		& \multicolumn{2}{c}{$\bm{M = 1500}$} 
		& \multicolumn{2}{c}{$\bm{M = 2400}$} 
		& \multicolumn{2}{c}{$\bm{M = 2700}$} \\
		\cmidrule(lr){2-3} \cmidrule(lr){4-5} \cmidrule(lr){6-7} \cmidrule(lr){8-9}
		\textbf{Algorithm} & \textbf{NI} & \textbf{T} 
		& \textbf{NI} & \textbf{T} 
		& \textbf{NI} & \textbf{T} 
		& \textbf{NI} & \textbf{T} \\
		\midrule
		FRB    & 1211 & 78.85 & 2715 & 182.38 & 6475 & 450.53 & 8682 & 612.17 \\
		FRBLSL & 992  & 64.65 & 2228 & 149.15 & 5320 & 370.22 & 7140 & 503.91 \\
		FRBLS & 751  & \textbf{48.89} & 1825 & 122.19 & 4370 & 303.76 & 5871 & 414.67 \\
		FRBLSR & 743  & 76.22 & 1506 & \textbf{103.64} & 3606 & \textbf{256.65} & 4842 & \textbf{348.11} \\
		\bottomrule
	\end{tabular}
	\caption{Results in terms of average number of iterations (NI) and average CPU time in seconds (T). For each dimension, the best CPU time is highlighted in \textbf{black}.}
	\label{T:results}
\end{table}
 
\subsection{Image Deblurring via $\ell^p$ Total Variation} Let $N \in \N$, $M \in \N$, $\H = \R^N$, $\G= \R^M$, and consider $b \in \R^M$ a blurred and noisy observation of an image $\overline{x} \in \R^N$. In particular, we assume that
\begin{equation*}
	b = T\overline{x}+ \epsilon,
\end{equation*}
where $T \in \R^{M\times N}$ is a linear operator modeling the blur process and $\epsilon$ represents a Gaussian noise. A popular model for recovering $\overline{x}$ is to solve the following optimization problem: 
\begin{equation}\label{eq:problemoCT}
	\min_{x\in [0,255]^N} F(x):=\frac{1}{2}\|Tx-b\|_2^2 +\frac{\lambda}{p} \|\nabla x\|_p^p,
\end{equation}
where $\lambda \in \RPP$ is a regularization parameter, $p \in \RPP$, $\|\cdot\|_p$ denotes the $\ell^p$ norm, and $\nabla$ is the discrete gradient.
Note that, when $p= 2$,  $\|\cdot\|_p^p$ is a convex function with Lipschitz continuous gradient, but when $p \in ]0,2[$ the gradient is no longer Lipschitz continuous, when $p \in ]0,1]$ the function is no longer differentiable, and when $p \in ]0,1[$ it is no longer convex. Some of these particular cases have been studied, for example, in \cite{Chambolle2004l1unified,BucciniLothar2020,Chambolle2004,Lanza2017,Lanza2015,ROF1992}. To test the numerical performance of FHRB with linesearch, we will focus on the case where $p \in ]1,2[$. We therefore solve the optimization problem in  \eqref{eq:problemoCT} by using Algorithm~\ref{alg:algoMTE} with $A=N_{[0,255]^N}$, $B= -\lambda \nabla^{\top} \left( |\nabla x|^{p-2} \nabla x \right)$, $C = T^\top(T(\cdot) -b)$, and $D=0$. Note that, in this same setting, the problem can be also solved by FBHF \cite{BricenoDavis2018}, we hence compare both algorithms. We test the performance of both algorithms with and without restarting the step-size at each linesearch. We refer to FHRB with restart as FHRBLSR and without restart as FHRBLS. Similarly, we use FBHFLSR and FBHFLS to refer to FHBH with and without restart, respectively.

In our experiments, the linear operator $T$ represents an average blur kernel of size $K \times K$ with symmetric boundary conditions for $K \in \{3,6,9\}$, implemented in MATLAB using the \texttt{imfilter} function. The observations were corrupted by additive zero-mean white Gaussian noise $\epsilon$ with a standard deviation of $20$. As test images, we consider the ones shown in Figure~\ref{fig:x256} and in Figure~\ref{fig:x512}, which have resolutions of $256\times 256$ and $512\times 512$ pixels, respectively. In all instances, the regularization parameter was set to $\lambda = 5$. To initialize the algorithms, we performed a grid search over the parameter space $(\varepsilon, \tau) \in \{0.1, 0.2, \dots, 0.9\}^2$, alongside the parameters $\theta$ and $\gamma_0$ defined in Table~\ref{Tab:IP} (noting that $C$ is $\beta$-cocoercive with $\beta = 1/\|T\|^2$, $L=0$, and we choose $\kappa^2=\rho =  (2\theta^2)^{1/3}$, see \eqref{eq:stepsize2} and Remark~\ref{remthetarho}). For these tests, we set $p = 1/2$. Table~\ref{T:resimages1} reports the best performance achieved for each algorithm over the grid of $(\varepsilon, \tau)$. From this table, we observe that the algorithms incorporating step-size restarting terminate in fewer iterations; however, this does not necessarily translate to a lower CPU time. This behavior is attributable to the fact that algorithms without restarting allow for larger step-sizes, which, in some instances, require additional subiterations to satisfy the linesearch condition. Furthermore, the algorithms without restarting appear to be more robust with respect to the algorithmic parameters $\varepsilon$ and $\tau$, as the optimal values yielding the best results remain constant even when the dimension of the problem or the blur operator changes. From Table~\ref{T:resimages1} we also observed that for the $3\times 3$ kernels, FBHFLS performs best in terms of CPU time, although FHRBLS remains highly competitive. For the $6\times 6$ kernels, the fastest CPU times are achieved by FHRBLSR and FBHFLS for the $256\times 256$ and $512\times 512$ images, respectively. For the $9\times 9$ kernel, the FHRB type methods dominate and attain the lowest CPU times for the $256\times 256$ and $512\times 512$ images, respectively. With respect to the final objective function values, FHRBLSR attains the lowest minimum across all tested scenarios. Table~\ref{T:res_p_variation} presents the results for $p=1.2$ and $p=1.8$ using the $256\times 256$ pixel image. From this table, we observe that for $p=1.2$, the algorithms utilizing step-size restarting exhibit worse performance in terms of CPU time. This behavior can be attributed to the fact that for $1 < p < 2$, the local Lipschitz constant of the gradient of $\|\cdot\|_p^p$ grows unbounded near the origin, with this growth becoming even more pronounced as $p$ approaches $1$. Consequently, when the step-size is reinitialized to a large value, the linesearch procedure requires a significantly higher number of subiterations to satisfy the linesearch condition. Notice that for $p=1.2$ and across every kernel size, the FHRBLS methods outperform the other algorithms. On the other hand, for $p=1.8$, the methods with and without restarting demonstrate comparable performance, yielding results similar to the $p=3/2$ case. Some of the reconstructed images are shown in Figure~\ref{fig:256k3}, Figure~\ref{fig:512k6}, and Figure~\ref{fig:512k9}.

 Finally, we conclude that the algorithms without restart are more robust regarding changes in dimensions, kernel size, and the value of $p$. However, a further analysis of the convergence rates or the linesearch procedure for H\"older continuous operators should be considered in future research.
\begin{table}[h!]
    \centering  
    \begin{tabular}{lcc}
        \toprule
        \textbf{Algorithm} & $\gamma_0$ & $\theta$ \\
        \midrule
        FHRBLSR/FHRBLS & $2\varepsilon \tau/\|T\|^2$ & $\theta^2+3{\left(\frac{\theta}{2}\right)}^{\frac{2}{3}}=0.99(1-\varepsilon)$\\
        FBHFLSR/FBHFLS & $2\varepsilon \tau/\|T\|^2$ & $0.99\sqrt{1-\varepsilon}$ \\
        \bottomrule
    \end{tabular}    \caption{Initialization parameters for FHRBLSR/FHRBLS and FBHFLSR/FBHFLS.}  \label{Tab:IP}
\end{table}

\begin{table}[h!]
	\centering
    \vspace{0.15cm} 
    \resizebox{\textwidth}{!}{
	\begin{tabular}{l ccccc ccccc ccccc }
		\toprule
		$256 \times 256$ & \multicolumn{5}{c}{Kernel $3\times 3$} & 
		\multicolumn{5}{c}{Kernel $6\times 6$} & \multicolumn{5}{c}{Kernel $9\times 9$}  \\
		\cmidrule(lr){2-6} \cmidrule(lr){7-11} \cmidrule(lr){12-16}
		\textbf{Algorithm}  & $\varepsilon$ & $\tau$ & \textbf{NI} & \textbf{T} & $F(x_n)$ 
	& $\varepsilon$ & $\tau$ & \textbf{NI} & \textbf{T} & $F(x_n)$ 
	& $\varepsilon$ & $\tau$ & \textbf{NI} & \textbf{T} & $F(x_n)$ \\
		\midrule
		FHRBLSR    & 0.2 & 0.1 & 554 & 0.86 & 12.74 & 0.2 & 0.1 & 615  & \textbf{1.05} & 13.15 & 0.2 & 0.1 & 760  & 1.29 & 13.20 \\
		FHRBLS    & 0.1 & 0.6 & 860 & 0.69 & 12.76 & 0.1 & 0.4 & 1258 & 1.23 & 13.17 & 0.5 & 0.2 & 1350 & \textbf{1.21} & 13.20 \\
		FBHFLSR    & 0.3 & 0.1 & 367 & 0.80 & 12.75 & 0.3 & 0.1 & 597  & 1.43 & 13.17 & 0.2 & 0.1 & 858  & 1.82 & 13.21 \\
		FBHFLS    & 0.3 & 0.1 & 427 & \textbf{0.63} & 12.75 & 0.2 & 0.4 & 779  & 1.21 & 13.18 & 0.2 & 0.1 & 1000 & 1.47 & 13.21 \\		
		\midrule
        $512 \times 512$ & \multicolumn{5}{c}{Kernel $3\times 3$} & 
		\multicolumn{5}{c}{Kernel $6\times 6$} & \multicolumn{5}{c}{Kernel $9\times 9$}  \\
		\cmidrule(lr){2-6} \cmidrule(lr){7-11} \cmidrule(lr){12-16}
		\textbf{Algorithm}  & $\varepsilon$ & $\tau$ & \textbf{NI} & \textbf{T} & $F(x_n)$ 
	& $\varepsilon$ & $\tau$ & \textbf{NI} & \textbf{T} & $F(x_n)$ 
	& $\varepsilon$ & $\tau$ & \textbf{NI} & \textbf{T} & $F(x_n)$ \\
		\midrule
		FHRBLSR    & 0.2 & 0.1 & 609 & 10.27 & 51.88 & 0.2 & 0.1 & 735  & 13.58 & 53.36 & 0.2 & 0.1 & 794  & \textbf{14.44} & 53.35 \\
		FHRBLS    & 0.1 & 0.7 & 911 & 8.05  & 51.96 & 0.1 & 0.4 & 1411 & 12.78 & 53.42 & 0.1 & 0.4 & 1579 & 14.45 & 53.35 \\
		FBHFLSR    & 0.3 & 0.1 & 391 & 9.86  & 51.90 & 0.3 & 0.1 & 600  & 16.00 & 53.40 & 0.3 & 0.1 & 844  & 22.03 & 53.35 \\
		FBHFLS    & 0.3 & 0.1 & 462 & \textbf{8.00}  & 51.90 & 0.3 & 0.1 & 706  & \textbf{11.50} & 53.38 & 0.4 & 0.4 & 1165 & 18.76 & 53.35 \\        		\bottomrule
	\end{tabular}
    } \caption{Numerical results for $p=1.5$ in terms of average number of iterations (NI), average CPU time in seconds (T), and the value of the objective function (scaled by $10^{-6}$). For each dimension, the best CPU time is highlighted in \textbf{black}.}
	\label{T:resimages1}
\end{table}

\begin{table}[h!]
	\centering
    \vspace{0.15cm}
    \resizebox{\textwidth}{!}{
	\begin{tabular}{l ccccc ccccc ccccc }
		\toprule
		$p = 1.2$ & \multicolumn{5}{c}{Kernel $3\times 3$} & 
		\multicolumn{5}{c}{Kernel $6\times 6$} & \multicolumn{5}{c}{Kernel $9\times 9$}  \\
		\cmidrule(lr){2-6} \cmidrule(lr){7-11} \cmidrule(lr){12-16}
		\textbf{Algorithm}  & $\varepsilon$ & $\tau$ & \textbf{NI} & \textbf{T} & $F(x_n)$ 
	& $\varepsilon$ & $\tau$ & \textbf{NI} & \textbf{T} & $F(x_n)$ 
	& $\varepsilon$ & $\tau$ & \textbf{NI} & \textbf{T} & $F(x_n)$ \\
		\midrule
		FHRBLSR    & 0.2 & 0.1 & 1391 & 9.94 & 11.98 & 0.2 & 0.1 & 1517 & 10.82 & 13.03 & 0.2 & 0.1 & 1606 & 11.01 & 13.21 \\
		FHRBLS    & 0.2 & 0.1 & 2019 & \textbf{5.17} & 11.96 & 0.3 & 0.2 & 2457 & \textbf{6.38} & 13.04 & 0.2 & 0.1 & 2217 & \textbf{5.45} & 13.22 \\
		FBHFLSR    & 0.3 & 0.1 & 1482 & 13.32 & 12.20 & 0.3 & 0.1 & 1505 & 13.45 & 13.04 & 0.3 & 0.1 & 1471 & 12.76 & 13.18 \\
		FBHFLS    & 0.3 & 0.1 & 1482 & 6.55 & 12.20 & 0.3 & 0.1 & 1505 & 6.70 & 13.04 & 0.3 & 0.1 & 1471 & 6.28 & 13.18 \\		
		\midrule
        $p = 1.8$ & \multicolumn{5}{c}{Kernel $3\times 3$} & 
		\multicolumn{5}{c}{Kernel $6\times 6$} & \multicolumn{5}{c}{Kernel $9\times 9$}  \\
		\cmidrule(lr){2-6} \cmidrule(lr){7-11} \cmidrule(lr){12-16}
		\textbf{Algorithm}  & $\varepsilon$ & $\tau$ & \textbf{NI} & \textbf{T} & $F(x_n)$ 
	& $\varepsilon$ & $\tau$ & \textbf{NI} & \textbf{T} & $F(x_n)$ 
	& $\varepsilon$ & $\tau$ & \textbf{NI} & \textbf{T} & $F(x_n)$ \\
		\midrule
		FHRBLSR    & 0.1 & 0.1 & 290 & 0.98 & 13.65 & 0.1 & 0.1 & 302 & \textbf{1.12} & 13.68 & 0.1 & 0.1 & 377 & 1.48 & 13.53 \\
		FHRBLS    & 0.2 & 0.1 & 391 & 0.97 & 13.46 & 0.5 & 0.2 & 470 & 1.22 & 13.67 & 0.2 & 0.1 & 547 & \textbf{1.37} & 13.47 \\
		FBHFLSR    & 0.1 & 0.1 & 226 & 1.17 & 13.65 & 0.1 & 0.1 & 332 & 1.78 & 13.68 & 0.2 & 0.2 & 329 & 2.24 & 13.54 \\
		FBHFLS    & 0.1 & 0.1 & 198 & \textbf{0.85} & 13.64 & 0.8 & 0.1 & 259 & 1.16 & 13.68 & 0.2 & 0.2 & 322 & 1.41 & 13.53 \\        
		\bottomrule
	\end{tabular}
    }\caption{Numerical results for the $256 \times 256$ image and $p = 1.2$ and $p = 1.8$. Objective function values are scaled by $10^{-6}$. For each setting, the best CPU time is highlighted in \textbf{black}.}
	\label{T:res_p_variation}
\end{table}


\begin{figure}
\centering
\subfloat[\scriptsize $256\times 256$ pixels]{\label{fig:x256}\includegraphics[scale=0.3]{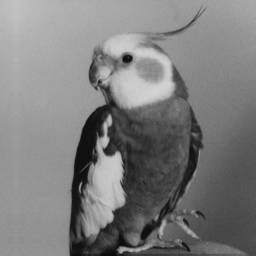}}\,
\subfloat[\tiny Blur/Noisy (22.00)]{\label{b_256_k3}\includegraphics[scale=0.3]{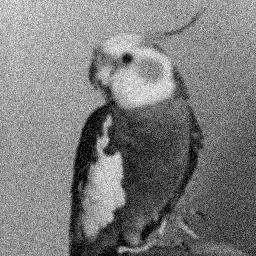}}\\
\raisebox{1.3cm}{\makebox[0.05\textwidth]{$p=1.2$ \qquad}}
\subfloat[\tiny FHRBLSR (25.79)]{\label{x_FRB1_rho_N12_256_k3}\includegraphics[scale=0.3]{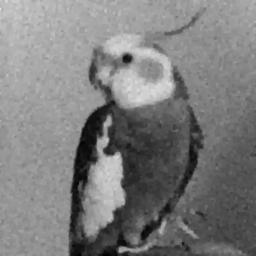}}\,
\subfloat[\tiny FHRBLS (25.80)]{\label{x_FRB2_rho_N12_256_k3}\includegraphics[scale=0.3]{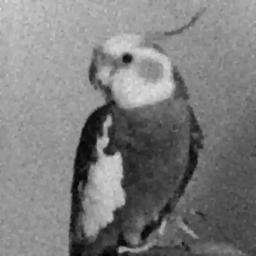}}\,
\subfloat[\tiny FBHFLSR (25.80)]{\label{x_FBF1_rho_N12_256_k3}\includegraphics[scale=0.3]{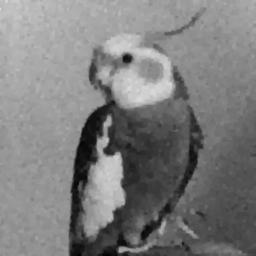}}\,
\subfloat[\tiny FBHFLS (25.81)]{\label{x_FBF2_rho_N12_256_k3}\includegraphics[scale=0.3]{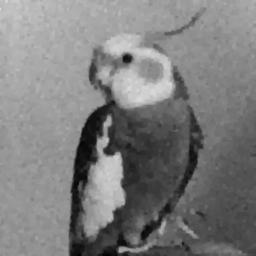}}\\
\raisebox{1.3cm}{\makebox[0.05\textwidth]{$p=1.5$ \qquad}}
\subfloat[\tiny FHRBLSR (31.36)]{\label{x_FRB1_rho_N15_256_k3}\includegraphics[scale=0.3]{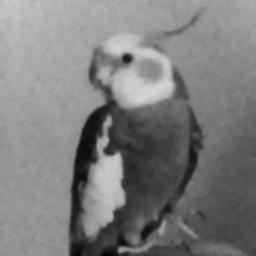}}\,
\subfloat[\tiny FHRBLS (31.32)]{\label{x_FRB2_rho_N15_256_k3}\includegraphics[scale=0.3]{x_FRB2_rho_N15_256_k3.jpg}}\,
\subfloat[\tiny FBHFLSR (31.35)]{\label{x_FBF1_rho_N15_256_k3}\includegraphics[scale=0.3]{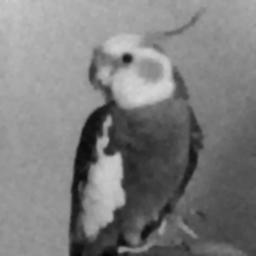}}\,
\subfloat[\tiny FBHFLS (31.31)]{\label{x_FBF2_rho_N15_256_k3}\includegraphics[scale=0.3]{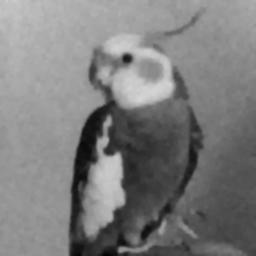}}\\
\raisebox{1.3cm}{\makebox[0.05\textwidth]{$p=1.8$ \qquad}}
\subfloat[\tiny FHRBLSR (29.37)]{\label{x_FRB1_rho_N18_256_k3}\includegraphics[scale=0.3]{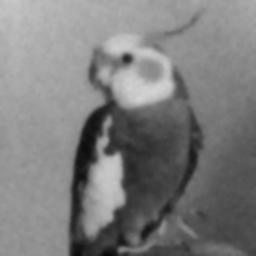}}\,
\subfloat[\tiny FHRBLS (29.67)]{\label{x_FRB2_rho_N18_256_k3}\includegraphics[scale=0.3]{x_FRB2_rho_N15_256_k3.jpg}}\,
\subfloat[\tiny FBHFLSR (29.37)]{\label{x_FBF1_rho_N18_256_k3}\includegraphics[scale=0.3]{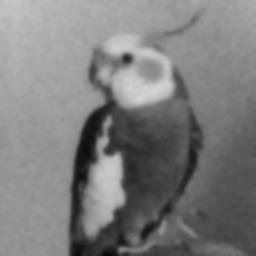}}\,
\subfloat[\tiny FBHFLS (29.38)]{\label{x_FBF2_rho_N18_256_k3}\includegraphics[scale=0.3]{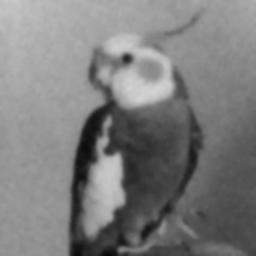}}
\caption{Restored images for $256\times 256$ pixels, $3\times 3$ kernel, and $p \in \{1.2,1.5,1.8\}$. The values in parentheses represent the PSNR of the displayed image relative to the original image.}\label{fig:256k3} 
\end{figure}

\begin{figure}
\centering
\subfloat[\scriptsize $512\times 512$ pixels]{\label{fig:x512_k6}\includegraphics[scale=0.19]{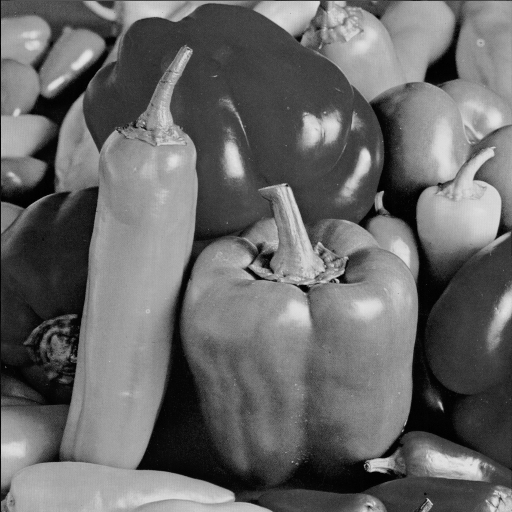}}\,
\subfloat[\tiny Blur/Noisy (20.91)]{\label{b_512_k6}\includegraphics[scale=0.19]{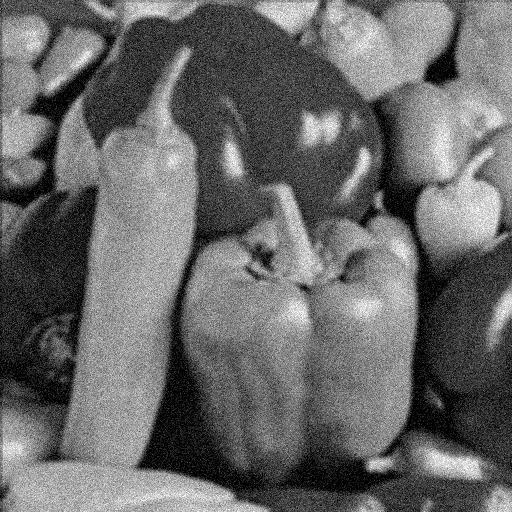}}\\
\subfloat[\tiny FHRBLSR (26.96)]{\label{x_FRB1_rho_N15_512_k6}\includegraphics[scale=0.19]{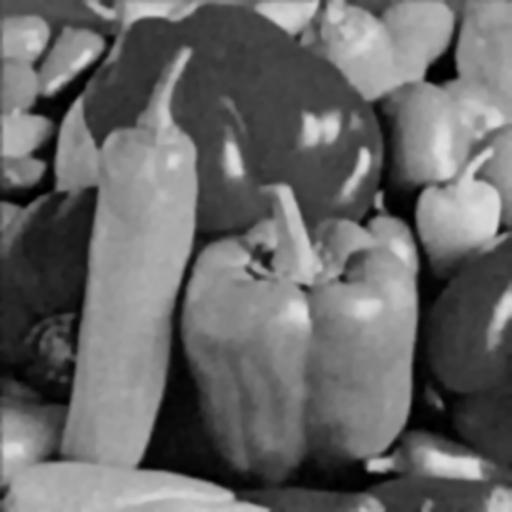}}\,
\subfloat[\tiny FHRBLS (26.99)]{\label{x_FRB2_rho_N15_512_k6}\includegraphics[scale=0.19]{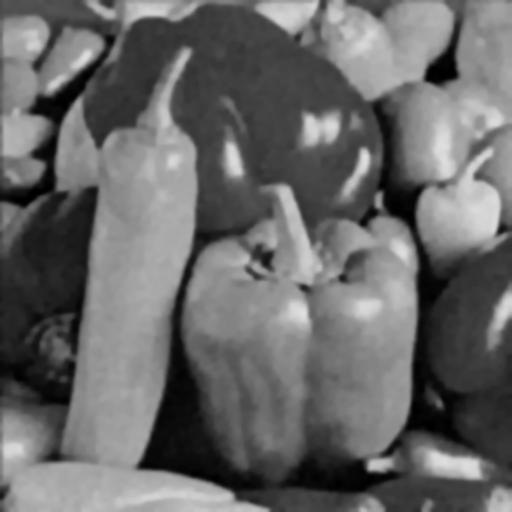}}\,
\subfloat[\tiny FBHFLSR (26.95)]{\label{x_FBF1_rho_N15_512_k6}\includegraphics[scale=0.19]{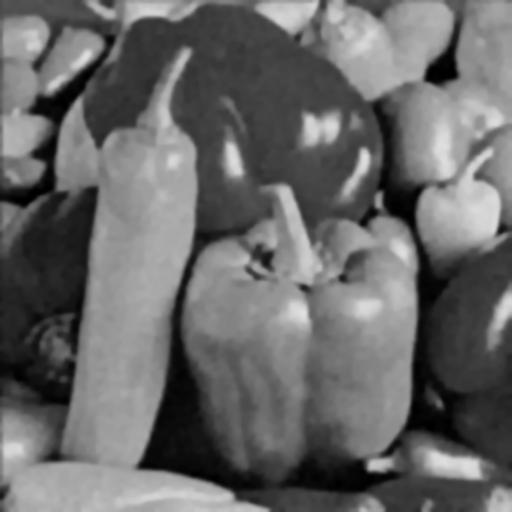}}\,
\subfloat[\tiny FBHFLS (26.95)]{\label{x_FBF2_rho_N15_512_k6}\includegraphics[scale=0.19]{x_FBF2_rho_N15_512_k6.jpg}}
\caption{Restored images for $512\times 512$ pixels, $6\times 6$ kernel, and $p =1.5$. The values in parentheses represent the PSNR of the displayed  image relative to the original image.}\label{fig:512k6} 
\end{figure}

\begin{figure}
\centering
\subfloat[\scriptsize $512\times 512$ pixels]{\label{fig:x512}\includegraphics[scale=0.19]{x_original512.png}}\,
\subfloat[\tiny Blur/Noisy (20.51)]{\label{b_512_k3}\includegraphics[scale=0.19]{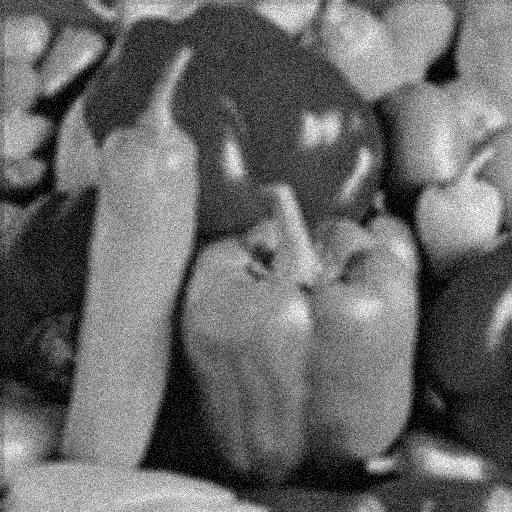}}\\
\subfloat[\tiny FHRBLSR (25.81)]{\label{x_FRB1_rho_N15_512_k9}\includegraphics[scale=0.19]{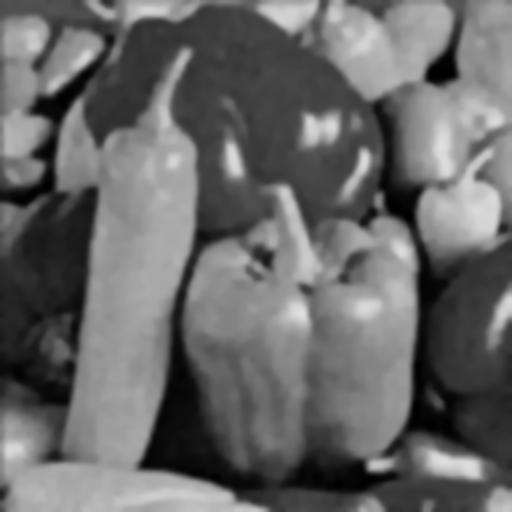}}\,
\subfloat[\tiny FHRBLS (25.78)]{\label{x_FRB2_rho_N15_512_k9}\includegraphics[scale=0.19]{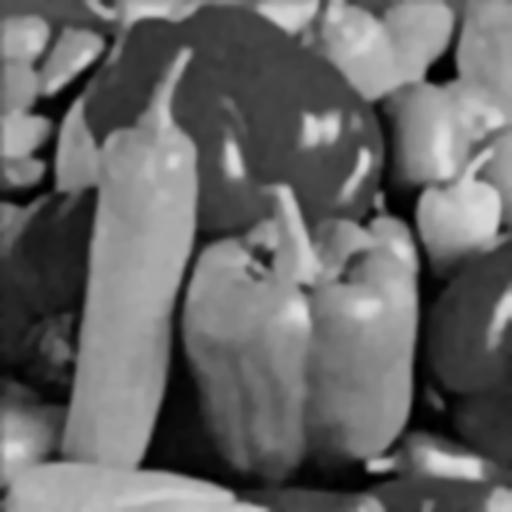}}\,
\subfloat[\tiny FBHFLSR (25.82)]{\label{x_FBF1_rho_N15_512_k9}\includegraphics[scale=0.19]{x_FBF2_rho_N15_512_k6.jpg}}\,
\subfloat[\tiny FBHFLS (25.82)]{\label{x_FBF2_rho_N15_512_k9}\includegraphics[scale=0.19]{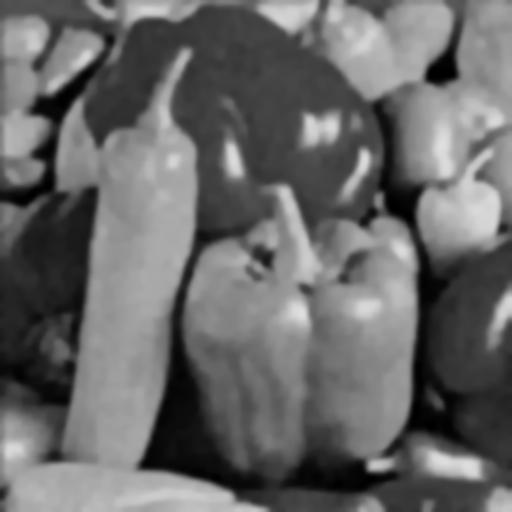}}
\caption{Restored images for $512\times 512$ pixels, $9\times 9$ kernel, and $p =1.5$. The values in parentheses represent the PSNR of the displayed image relative to the original image.}\label{fig:512k9} 
\end{figure}

\section{Conclusion}\label{sec:conclu}
In this article, we studied the convergence of the forward-reflected-backward algorithm for solving monotone inclusions involving merely continuous operators. We included an example demonstrating that the existing linesearch for locally Lipschitz continuous operators fails to terminate when the operator is only continuous. Motivated by this example, we proposed a new linesearch strategy that is guaranteed to terminate in a finite number of steps, and we proved the weak convergence of the proposed method to a solution. In addition, we extended the proposed method to a four-operator splitting scheme capable of solving inclusions that additionally incorporate cocoercive and Lipschitz continuous operators, applying this framework to convex optimization problems with nonlinear constraints. Finally, we provided numerical experiments on saddle-point and image deblurring problems, demonstrating the computational advantages of the proposed methods.
\section*{Acknowledgment}
The authors were partially supported by ANID through FONDECYT Iniciación Grant 11250164.

\end{document}